\documentclass{amsart}

\usepackage{amsmath}%
\usepackage{amsfonts}%
\usepackage{amssymb}%
\usepackage{graphicx}
\usepackage{subfigure}

\newtheorem{thm}{Theorem}[section]
\theoremstyle{plain}

\newtheorem{lemma}[thm]{Lemma}

\theoremstyle{remark}
\newtheorem{rem}[thm]{Remarks}

\newcommand{\Tone}{\mathbb{T}}
\newcommand{\rs}{r_{\star}}
\newcommand{\trho}{\tilde{\rho}}
\newcommand{\brho}{\bar{r}}
\newcommand{\bG}{\bar{\mathcal{G}}}
\newcommand{\cH}{\mathcal{H}}
\newcommand{\tr}{\tilde{r}}

\numberwithin{equation}{section}

\begin{document}
\title[AMC Stability and Bifurcation]{Stability and Bifurcation of Equilibria for the Axisymmetric Averaged Mean Curvature Flow}
\author{Jeremy LeCrone}
\address
{Department of Mathematics \newline%
\indent Vanderbilt University, Nashville, TN USA}%
\email{jeremy.lecrone@vanderbilt.edu}%
%\author{Gieri Simonett}
%\email[G. Simonett]{gieri.simonett@vanderbilt.edu}%
%\date{March 15, 2002}

\begin{abstract}
We study the \emph{averaged mean curvature flow}, also called the 
\emph{volume preserving mean curvature flow}, in the particular setting of axisymmetric surfaces
embedded in $\mathbb{R}^3$ satisfying periodic boundary conditions. We establish analytic well--posedness 
of the flow within the space of little-H\"older continuous surfaces, given rough initial data. 
We also establish dynamic properties of equilibria, including stability, instability, and bifurcation behavior
of cylinders, where the radius acts as a bifurcation parameter.
\end{abstract}
\maketitle

%%%%%%%%%%%%%%%%%%%%%%%%%%%%%%%%%%%%%%%%%%%%%%%%%%%%%%%%%%%%%%%%%%%%
%%%%%%%%                      CONTENT                        %%%%%%%
%%%%%%%%%%%%%%%%%%%%%%%%%%%%%%%%%%%%%%%%%%%%%%%%%%%%%%%%%%%%%%%%%%%%

%%%%%%%%%%%%%%%%%%%%%%%%%%%%%%%%%%%%%%%%%%%%%%%%%%%%%%%%%%%%%%%%%%%%%%%%%%%%%%%%%%%%%%%%%%%%%%%%%%%%%%%%%%%%%%%%%%%
\section{Introduction}
%%%%%%%%%%%%%%%%%%%%%%%%%%%%%%%%%%%%%%%%%%%%%%%%%%%%%%%%%%%%%%%%%%%%%%%%%%%%%%%%%%%%%%%%%%%%%%%%%%%%%%%%%%%%%%%%%%%

The averaged mean curvature flow is a second--order geometric evolution law, acting on closed, compact,
connected, sufficiently smooth hypersurfaces $\Gamma$ immersed in $\mathbb{R}^n$. The evolution of $\Gamma$ 
involves the (normalized) mean curvature $\cH(\Gamma)$, 
which is simply the sum of the principal curvatures on the surface. In particular, one seeks 
to find a one--parameter family of smooth, immmersed, orientable hypersurfaces $\{ \Gamma(t) : t \ge 0 \}$ which 
satisfy the evolution equation
\begin{equation}\label{MCF}
\begin{cases} V(t) = h(\Gamma(t)) - \cH(\Gamma(t)), &\text{$t > 0,$}\\ \Gamma(0) = \Gamma_0, \end{cases}
\end{equation}
where $V(t)$ denotes the normal velocity of the hypersurface $\Gamma(t)$, $\Gamma_0$ is a given initial surface, and
$h(\Gamma)$ is the integral average of the mean curvature $\cH(\Gamma)$ on $\Gamma$. Two important features of the
flow \eqref{MCF} are that the surface area of $\Gamma(t)$ is decreasing in $t$ and the (signed) volume enclosed by 
the surface $\Gamma(t)$ is preserved, as long as smooth solutions exist. These features provide a starting
point for our analysis of the dynamical properties of solutions.

Problem \eqref{MCF} is a modification of the mean curvature flow, which is the related evolution equation  
without the integral average term $h(\Gamma)$ in the governing equation. These two problems have a long and 
rich history of investigation which we do not attempt to summarize here. We will simply 
highlight some of the important historical results which bear on our current investigation. 

Both the mean curvature and averaged mean curvature 
flows have been considered in a wide range of settings for $\Gamma$. 
Cases considered include curves in $\mathbb{R}^2$, \cite{GAG86, GH86}, 
hypersurfaces immersed in $\mathbb{R}^n$, \cite{ATH97, ES98, HUI87, MS00},
and immersed submanifolds of smooth Riemannian manifolds $(\mathcal{M}^n, g)$,
\cite{HP99}, for arbitrary dimensions $n \ge 2$. 

The first well--posedness results for \eqref{MCF} were established by Gage \cite{GAG86}, 
in the setting of plane curves $\gamma \subset \mathbb{R}^2$.
Under convexity assumptions on the initial curve $\gamma_0$, 
it is shown that solutions are global, 
convex for all time $t > 0,$ 
and solutions converge to a circle.
These results were later generalized by Huisken \cite{HUI87}, 
who proved existence of global solutions assuming 
smooth uniformly convex initial hypersurfaces in $\mathbb{R}^n$.
It is also shown in this setting that solutions converge to an $n$--dimensional sphere. 
Escher and Simonett \cite{ES98} later proved well--posedness
for rough initial data (in the little H\"older spaces $h^{1 + \beta}$), 
including the possibility of nonconvex hypersurfaces, and further proved that the family
of $n$--dimensional spheres is locally exponentially attractive in the topology of $h^{1 + \beta}$. 
Regarding further qualitative properties of solutions, 
Mayer and Simonett \cite{MS00} proved the first analytic results regarding
hypersurfaces evolving according to \eqref{MCF} which lose embeddedness.

Considering axisymmetric, or rotationally symmetric, surfaces, 
several authors have investigated both the mean curvature and 
averaged mean curvature flows in this setting, c.f. \cite{ATH97, CRM09}
and references therein. 
Of particular note is a result of Athanassenas \cite{ATH97}, 
in the setting of smooth surfaces satisfying Neumann boundary conditions,
regarding global existence and convergence to a cylinder for given
initial surfaces satisfying an isoperimetric--type inequality. 
A recent article of Hartley \cite{HAR12} refines this result to rough initial data, 
also allowing for non--rotationally symmetric initial
surfaces. An exponential convergence rate is also established, 
given initial data $h^{1 + \beta}$--close to a cylinder, 
using techniques similar to \cite{ES98}. 

In the current article, we develop a function space setting adequate 
to establish well--posedness of the axisymmetric averaged mean curvature 
flow with rough initial data and to investigate geometric properties
of solutions analytically. The methods of the paper also allow for application
to higher--order nonlinear problems. For example the author and
Simonett \cite{LS12} recently applied similar methods to derive
results for the surface diffusion flow, a fourth--order geometric
evolution law, also in the setting of axisymmetric surfaces 
with periodic boundary conditions.

We conclude the introduction with a brief outline of the article.
In Section~\ref{Sec:Notation}, we define appropriate spaces and 
introduce analytic tools that will be used throughout the paper. 
In Section~\ref{Sec:WellPosedness}, we take advantage of
maximal regularity properties for the governing operator of \eqref{AMC}, 
in conjunction with a quasilinear structure, 
in order to establish well--posedness and regularity of solutions. 
We are then able to explicitly characterize all equilibria 
(stationary solutions) to the problem in Section~\ref{Sec:Equilibria}, 
using results regarding closed surfaces of revolution with constant mean curvature. 

After characterizing equilibria, the remainder of the paper is 
dedicated to investigating properties of two particular families of equilibria, 
namely the cylinders of radius $\rs > 0$ and families 
of $\frac{2 \pi}{k}$--periodic unduloids. In Section~\ref{Sec:Stability}, 
we develop nonlinear stability and instability results for cylinders, 
where the size of the radius governs the stability of the cylinder. 
One will see in this section how maximal regularity provides a crucial 
connection to tools in nonlinear functional analysis,
including applications of the Implicit Function Theorem on Banach spaces. 
We conclude the article with an investigation of bifurcations of equilibria, 
in particular we use analytic methods to establish the qualitative structure of 
the intersections between the families of unduloids and the family of cylinders.

%%%%%%%%%%%%%%%%%%%%%%%%%%%%%%%%%%%%%%%%%%%%%%%%%%%%%%%%%%%%%%%%%%%%%%%%%%%%%%%%%%%%%%%%%%%%%%%%%%%%%%%%%%%%%%%%%%%
\section{Notation and Conventions}\label{Sec:Notation}
%%%%%%%%%%%%%%%%%%%%%%%%%%%%%%%%%%%%%%%%%%%%%%%%%%%%%%%%%%%%%%%%%%%%%%%%%%%%%%%%%%%%%%%%%%%%%%%%%%%%%%%%%%%%%%%%%%%

%%%%%%%%%%%%%%%%%%%%%%%%%%%%%%%%%%%%%%%%%%%%%%%%%%%%%%%%%%%%%%%%%%%%%%%%%%%%%%%%%%%%%%%%%%%%%%%%%%%%%%%%%%%%%%%%%%%
%\section{Axisymmetric Averaged Mean Curvature Flow}
%%%%%%%%%%%%%%%%%%%%%%%%%%%%%%%%%%%%%%%%%%%%%%%%%%%%%%%%%%%%%%%%%%%%%%%%%%%%%%%%%%%%%%%%%%%%%%%%%%%%%%%%%%%%%%%%%%%

For the remainder of the paper, we consider the case of $\Gamma \subset \mathbb{R}^3$ an embedded
surface which is symmetric about an axis of rotation (which we take to be the $x$--axis,
without loss of generality) and satisfies prescribed periodic boundary conditions on some fixed
interval $L$ of periodicity (we take $L = [-\pi, \pi]$ and enforce $2 \pi$ periodicity,
without loss of generality). In particular, the axisymmetric surface $\Gamma$  
is characterized by the parametrization
\[
\Gamma = \Big\{ (x, r(x) \cos(\theta), r(x) \sin(\theta)): \; x \in \mathbb{R}, \; \theta \in [-\pi, \pi] \Big\},
\]
where the function $r : \mathbb{R} \rightarrow (0, \infty)$ is the \emph{profile function}
for the surface $\Gamma$. Conversely, a profile function $r: \mathbb{R} \rightarrow (0, \infty)$
generates an axisymmetric surface $\Gamma = \Gamma(r)$ via the parametrization given above. 
Utilizing this explicit parametrization for axisymmetric surfaces, we can recast the averaged
mean curvature flow as an evolution equation for the time--dependent profile functions $r = r(t,x)$. 

Given a surface $\Gamma(r)$, it follows that the mean curvature is $\mathcal{H}(r) = \kappa_1 + \kappa_2$, where
\[
\kappa_1 = \frac{1}{r \sqrt{1 + r_x^2}} \quad \text{and} \quad \kappa_2 = \frac{- r_{xx}}{(1 + r_x^2)^{3/2}}
\]
are the \emph{azimuthal} and \emph{axial} principle curvatures, respectively. 
Meanwhile, the normal velocity of $\Gamma = \Gamma(t)$ is
\[
V(t) = \frac{r_t}{\sqrt{1 + r_x^2}},
\]
and, introducing the surface area functional 
\begin{equation}\label{Eqn:S}
S(r) := \int_{-\pi}^{\pi} r(x) \sqrt{1 + r_x^2(x)} dx,
\end{equation}
which measures the surface area of one period of the induced 
axisymmetric surface $\Gamma(r)$ (modulo a factor of $2 \pi$),
the integral average of the mean curvature is
\begin{equation}\label{Eqn:h}
h(r) = \frac{1}{S(r)} \int_{-\pi}^{\pi} \cH(r)(x) r(x) \sqrt{1 + r_x^2(x)} dx.
\end{equation} 
Defining the operator
\begin{equation}\label{Eqn:G}
G(r) := \sqrt{1 + r_x^2} \; \Big[ h(r) - \cH(r) \Big],
\end{equation}
we arrive at the expression
\begin{equation}\label{AMC}
\begin{cases} r_t(t,x) = G(r(t))(x), &\text{$x \in \Tone, \; t > 0$}\\ r(0) = r_0, &\text{$x \in \Tone,$} \end{cases}
\end{equation}
for the periodic axisymmetric averaged mean curvature flow,
where $\Tone := [-\pi, \pi]$ is the one-dimensional torus; with the points
$-\pi$ and $\pi$ identified, endowed with the topology generated
by the metric
\[
d_{\Tone}(x,y) := \min \{|x - y|, 2 \pi - |x - y| \}, \qquad x,y \in \Tone.
\]

There is a natural relation between functions defined on $\Tone$ and $2 \pi$--periodic functions on 
$\mathbb{R}$, which the author explores in detail in the article \cite{LeC11}. Throughout the article, we
will consider profile functions in the periodic little--H\"older spaces $h^{k + \alpha}(\Tone),$ which
are Banach spaces for $k \in \mathbb{N}_0 := \{ 0, 1, 2, \ldots \}$, $\alpha \in (0, 1)$, defined by
\begin{align*}
h^{k + \alpha}(\Tone) &:= \overline{C^{\infty} (\Tone)}^{\| \cdot \|_{C^{k + \alpha}(\Tone)}} \qquad \text{and}\\
\| r \|_{h^{k + \alpha}} &:= \| r \|_{C^{k + \alpha}(\Tone)} := \sum_{j = 0}^{k} \max_{x \in \Tone} |r^{(j)}(x)| + \sup_{\substack{x, y \in \Tone\\ x \not= y}} \frac{|r^{(k)}(x) - r^{(k)}(y)|}{d_{\Tone}^{\alpha}(x,y)},
\end{align*}
i.e. the closure of the smooth functions in the topology of the classic H\"older functions $C^{k + \alpha}$
over $\Tone$. If $\sigma \in \mathbb{R}_+ \setminus \mathbb{Z}$ is fixed, then we denote by 
$h^{\sigma}(\Tone)$ the little--H\"older space $h^{\lfloor \sigma \rfloor + \{ \sigma \}}(\Tone)$,
where $\lfloor \sigma \rfloor$ denotes the largest integer not exceeding $\sigma$ and 
$\{ \sigma \} := \sigma - \lfloor \sigma \rfloor$. 

In addition to functions defined on the one--dimensional torus, 
we consider standard classes of regular functions between Banach spaces. 
In particular, given Banach spaces $E,$$F$, and an open set $U \subset E$, 
we denote by $C^k(U,F)$ the class of $k$--times (Fr\'echet) 
differentiable functions mapping $U$ into $F$. We also denote by
$C^{\omega}(U,F)$ the class of real--analytic functions, which are
representable (in the topologies of $E$ and $F$) as a power series of $k$--linear maps
from $E$ into $F$, c.f. \cite{DE85}.

%%%%%%%%%%%%%%%%%%%%%%%%%%%%%%%%%%%%%%%%%%%%%%%%%%%%%%%%%%%%%%%%%%%%%%%%%%%%%%%%%%%%%%%%%%%%%%%%%%%%%%%%%%%%%%%%%%%
\subsection{Maximal Regularity}\label{MaxReg}
%%%%%%%%%%%%%%%%%%%%%%%%%%%%%%%%%%%%%%%%%%%%%%%%%%%%%%%%%%%%%%%%%%%%%%%%%%%%%%%%%%%%%%%%%%%%%%%%%%%%%%%%%%%%%%%%%%%

One essential tool that we use throughout the paper is the property of maximal 
regularity, also called \emph{optimal regularity} in the literature. Maximal regularity
has received considerable attention in connection with nonlinear
parabolic partial differential equations, 
c.f. \cite{AM95, AM05, AN90, CS01, LUN95, Pru03, PSZ09}. 
Although maximal regularity can be developed in a 
more general setting, we will focus on the setting of \emph{continuous} maximal regularity and 
direct the interested reader to the references \cite{AM95, LUN95} for a general development of the theory.

Let $\mu \in (0,1], \, J := [0,T]$, for some $T > 0$, and let $E$ be a (real or complex) Banach space. 
Following the notation of \cite{CS01}, we define spaces of continuous functions on $\dot{J} := J \setminus \{ 0 \}$ 
with prescribed singularity at 0 as
\begin{equation}\label{Eqn:SingularContinuity}
\begin{aligned}
&BU\!C_{1 - \mu}(J,E) := \bigg\{ u \in C(\dot{J}, E): [t \mapsto t^{1 - \mu} u(t)] \in BU\!C(\dot{J},E) \; \text{and} \\ 
& \hspace{9em} \lim_{t \rightarrow 0^+} t^{1 - \mu} \| u(t) \|_E = 0 \bigg\}, \quad \mu \in (0,1)\\
&\| u \|_{B_{1 - \mu}} := \sup_{t \in J} t^{1 - \mu} \| u(t) \|_E,
\end{aligned}
\end{equation}
where $BU\!C$ denotes the space consisting of bounded, uniformly continuous functions.
We also define the subspace
\[
BU\!C_{1 - \mu}^1(J,E) := \left\{ u \in C^1(\dot{J},E) : u, \dot{u} \in BU\!C_{1 - \mu}(J, E) \right\}, \quad \mu \in (0,1)
\]
and we set
\[
BU\!C_0 (J,E) := BU\!C(J,E), \qquad BU\!C^1_0(J,E) := BU\!C^1(J,E).
\]
If $J = [0, a)$ for $a > 0$, 
then we set
\begin{align*}
C_{1 - \mu}(J,E) &:= \{ u \in C(\dot{J},E): u \in BU\!C_{1 - \mu}([0,T],E), \quad T < \sup J \},\\
C^1_{1 - \mu}(J,E) &:= \{ u \in C^1(\dot{J},E): u, \dot{u} \in C_{1 - \mu}(J,E) \}, \qquad \mu \in (0,1],
\end{align*}
which we equip with the natural Fr\'echet topologies induced by 
$BU\!C_{1 - \mu}([0,T],E)$ and $BU\!C_{1 - \mu}^1([0,T],E)$, respectively.

If $E_1$ and $E_0$ are a pair of Banach spaces such that $E_1$ is continuously embedded in $E_0$,
denoted $E_1 \hookrightarrow E_0$, we set
\begin{equation}\label{Eqn:MaxRegSpaces}
\begin{split}
&\mathbb{E}_0(J) := BU\!C_{1 - \mu}(J,E_0), \qquad \mu \in (0,1],\\
&\mathbb{E}_1(J) := BU\!C^1_{1 - \mu}(J,E_0) \cap BU\!C_{1 - \mu}(J,E_1),
\end{split}
\end{equation}
where $\mathbb{E}_1(J)$ is a Banach space with the norm 
\[
\| u \|_{\mathbb{E}_1(J)} := \sup_{t \in \dot{J}} t^{1 - \mu} \Big( \| \dot{u}(t) \|_{E_0} + \| u(t) \|_{E_1} \Big).
\]
It follows that the trace operator $\gamma: \mathbb{E}_1(J) \rightarrow E_0$, defined by 
$\gamma v := v(0)$, is well-defined and we denote by $\gamma \mathbb{E}_1$ the image of 
$\gamma$ in $E_0$, which is itself a Banach space when equipped with the norm
\[
\| x \|_{\gamma \mathbb{E}_1} := \inf \Big\{ \| v \|_{\mathbb{E}_1(J)}: v \in \mathbb{E}_1(J) \, \text{and} \, \gamma v = x \Big\}.
\]

Given $B \in \mathcal{L}(E_1, E_0)$, 
closed as an operator on $E_0$, 
we say $\big( \mathbb{E}_0(J), \mathbb{E}_1(J) \big)$ 
is a \emph{pair of maximal regularity} for $B$ and write $B \in \mathcal{MR}_{\mu}(E_1, E_0)$, if 
\[
\left( \frac{d}{dt} + B, \, \gamma \right) \in \mathcal{L}_{isom}(\mathbb{E}_1(J), \mathbb{E}_0(J) \times \gamma \mathbb{E}_1), \qquad \mu \in (0, 1),
\]
where $\mathcal{L}_{isom}$ denotes the space of bounded linear isomorphisms. In particular,  
$B \in \mathcal{MR}_{\mu}(E_1, E_0)$ 
if and only if for every $(f, u_0) \in \mathbb{E}_0(J) \times \gamma \mathbb{E}_1$, there 
exists a unique solution $u \in \mathbb{E}_1(J)$ to the inhomogeneous Cauchy problem 
\[
\begin{cases}
\dot{u}(t) + Bu(t) = f(t), &t \in \dot{J},\\
u(0) = u_0.
\end{cases}
\]
Moreover, in the current setting, it follows that $\gamma \mathbb{E}_1 \, \dot{=} \, (E_0, E_1)_{\mu, \infty}^0$,
i.e. the trace space $\gamma \mathbb{E}_1$ is topologically equivalent to the noted continuous 
interpolation spaces of Da Prato and Grisvard, c.f. \cite{AM95, CS01, DPG79, LUN95}.

%%%%%%%%%%%%%%%%%%%%%%%%%%%%%%%%%%%%%%%%%%%%%%%%%%%%%%%%%%%%%%%%%%%%%%%%%%%%%%%%%%%%%%%%%%%%%%%%%%%%%%%%%%%%%%%%%%%
\section{Well--Posedness of \eqref{AMC}}\label{Sec:WellPosedness}
%%%%%%%%%%%%%%%%%%%%%%%%%%%%%%%%%%%%%%%%%%%%%%%%%%%%%%%%%%%%%%%%%%%%%%%%%%%%%%%%%%%%%%%%%%%%%%%%%%%%%%%%%%%%%%%%%%%

Well--posedness of the averaged mean curvature 
flow is well established in the literature, we reference the work of Escher and Simonett 
\cite{ES98} for fundamental local well--posedness in arbitrary space dimensions with rough 
initial data, and we also reference Anathanessas \cite{ATH97}, 
who considers rotationally symmetric surfaces in arbitrary space dimenions 
$\mathbb{R}^n$, satisfying Neumann boundary conditions. 
In the current periodic setting in $\mathbb{R}^3$, 
we establish the following well--posedness of \eqref{AMC}.

\begin{thm}\label{Thm:WellPosedness}
Fix $\alpha \in (0,1)$ and take $\mu \in [1/2, 1]$ so that $2 \mu + \alpha \notin \mathbb{Z}$.
For every initial value $r_0 \in V_{\mu} := h^{2 \mu + \alpha}(\Tone) \cap [ r > 0]$, there exists
a unique solution to \eqref{AMC},
\[
r(\cdot, r_0) \in C_{1 - \mu}^1(J(r_0), h^{\alpha}(\Tone)) \cap C_{1 - \mu}(J(r_0), h^{2 + \alpha}(\Tone)),
\]
on the maximal interval of existence $J(r_0) := [0, t^+(r_0))$. Further, we conclude that
\begin{enumerate}
	\item solutions have the additional regularity 
	\[
	r(\cdot, r_0) \in C^{\omega}((0, t^+(r_0)) \times \Tone) \qquad \text{for all} \quad r_0 \in V_{\mu}.
	\]
	
	\item \eqref{AMC} generates a real--analytic semiflow on $V_{\mu}$.
		%i.e. the time--dependent
		%semiflow $\phi^t(r_0) := r(t,r_0)$ depends analytically upon the initial data 
		%$r_0 \in \mathcal{D}_t := \{ r \in V_{\mu} : t^+(r) > t \}$, $t \in \mathbb{R}_+$. 
	
	\medskip
	
	\item if there exists $0 < M < \infty$ so that the bounds
		\begin{align*}
			\bullet& \quad r(t, r_0)(x) \ge 1/M, \quad \text{for all} \quad x \in \Tone, \quad \text{and}\\
			\bullet& \quad \| r(t,r_0) \|_{h^{\theta}} \leq M, \quad \text{for some} \quad \theta \in (1 + \alpha, \infty) \setminus \mathbb{Z},
		\end{align*}
		hold for all $t \in J(r_0),$ then $t^+(r_0) = \infty$.
\end{enumerate}
\end{thm}

To economize notation, we define the spaces, for $\alpha \in (0,1)$ fixed,
\begin{equation}\label{Eqn:LittleHolderE}
E_0 := h^{\alpha}(\Tone), \quad E_1 := h^{2 + \alpha}(\Tone), \quad \text{and} \quad
E_{\mu} := (E_0, E_1)_{\mu, \infty}^0, \quad \mu \in (0,1).
\end{equation}
Utilizing well--known results regarding the continuous interpolation
method $(\cdot, \cdot)_{\mu, \infty}^0$ and little--H\"older spaces,
we conclude 
\[
E_{\mu} = h^{2 \mu + \alpha}(\Tone) \quad \text{(up to equivalent norms),} \qquad \text{for} \quad 2 \mu + \alpha \notin \mathbb{Z},
\]
c.f. \cite{LeC11, LUN95}. 
Further, let $V$ be the set of functions $r: \Tone \rightarrow \mathbb{R}$ 
such that $r(x) > 0$ for all $x \in \Tone$ and define $V_{\mu} := V \cap E_{\mu}$ for $\mu \in [0,1]$.
Note that $V_{\mu}$ is an open subset of $E_{\mu}$ for all $\mu \in [0,1]$.

Before we prove the theorem, we reformulate the problem \eqref{AMC} in order to make explicit 
the quasilinear structure of the equation. In particular, we note that
\begin{align}\label{Eqn:QuasiStructure}
G(r) &= -\mathcal{A}(r) r + f(r) \quad \text{where} \quad (\mathcal{A}, f) : V_{1/2} \rightarrow \mathcal{L}(E_1, E_0) \times E_0,  \nonumber\\
\mathcal{A}(r) \rho &:= \frac{\sqrt{1 + r_x^2}}{S(r)} \int_{\Tone} \frac{r(x) \rho_{xx}(x)}{(1 + r_x^2(x))} \; dx - \frac{\rho_{xx}}{(1 + r_x^2)} \; , \\
f(r) &:= 2 \pi \; \frac{\sqrt{1 + r_x^2}}{S(r)} - \frac{1}{r} \; , \qquad \rho \in E_1, \; r \in V_{1/2} \,. \nonumber
\end{align}
%for $\rho \in h^{1 + \alpha}(\Tone)$ chosen arbitrarily.
%The regularity of $\mathcal{A}$ and $f$ is important for the well--posedness result, so we proceed with
%the following lemma which also establishes necessary maximal regularity properties of $\mathcal{A}(\rho)$.

\begin{lemma}\label{Lem:AandFRegularity}
Let $\mu \in [1/2, 1]$. Then
\[
(\mathcal{A}, f) \in C^{\omega} \bigg( V_{\mu} \,, \mathcal{MR}_{\nu}(E_1, E_0) \times E_0 \bigg), \qquad \nu \in (0,1].
\]
\end{lemma}

\begin{proof}
The regularity of the mappings is standard from techniques in nonlinear analysis, 
noting that every $r \in V_{\mu}$ is strictly positive on $\Tone$ and
$V_{\mu} \hookrightarrow h^{1 + \alpha}(\Tone)$ for $\mu \in [1/2, 1]$. Hence, we focus
on the maximal regularity result. In particular, for $r \in V_{\mu}$ fixed, it remains to 
show that $\mathcal{A}(r) \in \mathcal{MR}_{\nu}(E_1, E_0), \; \nu \in (0, 1].$ 

Separating the terms of the operator $\mathcal{A}(r),$ it follows from \cite[Theorem 5.2]{LeC11} 
that the second--order uniformly elliptic operator with variable coefficients,
\[
-\frac{1}{(1 + r_x^2)} \; \partial_x^2 \, ,
\]
is in the maximal regularity class $\mathcal{MR}_{\nu}(E_1, E_0), \; \nu \in (0, 1].$
Meanwhile, we note that
\begin{align*}
&\left\| \frac{\sqrt{1 + r_x^2}}{S(r)} \int_{\Tone} \frac{r(x)}{(1 + r_x^2(x))} \; \rho_{xx}(x) \; dx \right\|_{E_0} 
\le C(r) \int_{\Tone} | \rho_{xx}(x) | \; dx
\le 2 \pi \, C(r) \| \rho \|_{C^2(\Tone)}\\
&\hspace{3em}\Longrightarrow \quad \left[ \rho \mapsto \frac{\sqrt{1 + r_x^2}}{S(r)} \int_{\Tone} \frac{r(x)}{(1 + r_x^2(x))} \; \rho_{xx}(x) \, dx \right] \in \mathcal{L}(h^{2 + \beta}(\Tone), E_0),
\end{align*}
for any $0 < \beta < \alpha$, $r \in V_{1/2}$, where $C(r) > 0$ depends on the 
$h^{1 + \alpha}(\Tone)$--norm of $r$ and $\min_{x \in \Tone} r(x)$. 
Thus, from the interpolation inequality, in conjunction with the identity
$h^{2 + \beta}(\Tone) = (E_0, E_1)_{\theta, \infty}^0$, $\theta = 1 - \frac{\alpha - \beta}{2}$,
and an application of Young's inequality, 
we realize this term as a lower order perturbation, c.f. \cite[p. 2793]{ES98}.
The operator $\mathcal{A}(r)$ then satisfies maximal regularity
by a well--known perturbation result, c.f. \cite{AM95, ES98, LeC11}. 
%In fact, the result proceeds analogously to the argument in \cite[p.2793]{ES98}.
\end{proof}

\begin{proof}[Proof of Theorem~\ref{Thm:WellPosedness}]
For $\mu \in [1/2, 1),$ the main statement of the theorem, and parts {\em(b)} and {\em(c)}, follow from 
the results \cite{CS01} of Cl\'ement and Simonett regarding well--posedness of quasilinear problems, 
in the presence of sufficiently regular mappings $\mathcal{A}$ and $f$, and maximal regularity properties, 
as established in Lemma~\ref{Lem:AandFRegularity}. Meanwhile, when $\mu = 1$, we must consider \eqref{AMC}
as a fully nonlinear problem, for which DaPrato and Grisvard \cite{DPG79} have established well--posedness
and Angenent \cite{AN90} has established semiflow properties in the presence of 
maximal regularity properties for the Fr\'echet derivatives $DG(\rho)$.
Indeed, computing the Fr\'echet derivative $DG(\rho),$ one sees the structure  
\[
DG(r) \rho = \frac{\rho_{xx}}{(1 + r_x^2)} - \frac{\sqrt{1 + r_x^2}}{S(r)} \int_{\Tone} \frac{\rho_{xx}}{(1 + r_x^2)^{3/2}} \; dx + \mathcal{B}(r) \rho, 
\] 
for $\rho \in E_1$, where $\mathcal{B}(r)$ is a first--order differential operator. Hence, we apply
the same perturbation technique as in the proof of Lemma~\ref{Lem:AandFRegularity} to establish maximal
regularity properties for $DG(r), \; r \in V_{1}.$ 

To prove the additional regularity of solutions, part {\em (a)} of the theorem, we rely on a 
\emph{paramater--trick} that goes back to Masuda \cite{Mas80} and Angenent \cite{AN90, An90b},
where one can introduce parameters and apply the implicit function theorem to obtain 
regularity results for solutions, see also \cite{EPS03, ES96, LS12}.
In particular, we define the translation operator $T_a: \Tone \rightarrow \Tone$ which 
takes $x \in \Tone$ to the unique element of $\Tone$ which resides in the coset
$[x + a] \in \mathbb{R} / 2 \pi \mathbb{Z}$, for $a \in \mathbb{R}$.
Given a solution $r(\cdot, r_0)$ to \eqref{AMC}, for some $r_0 \in V_{\mu}$,
we take $t_1 \in (0, t^+(r_0))$ and consider the function
\[
r_{\lambda, a}(t, x) := r((1 + \lambda)t, r_0)(T_{ta} x), \qquad t \in I := [0, t_1], \, x \in \Tone,
\]
for $(\lambda, a) \in (- \delta, \delta)^2$, with $\delta > 0$ chosen sufficiently small.
It follows that $r_{\lambda, a} \in \mathbb{E}_1(I)$  and is a zero of the operator 
$\Phi: (\mathbb{E}_1(I) \cap C(I, V)) \times (-\delta, \delta)^2 \rightarrow \mathbb{E}_0(I) \times E_{\mu},$
\[
\Phi(v, (\lambda, a)) = (\partial_t v - (1 + \lambda)G(v) - a \partial_x v, \gamma v - r_0),
\]
where the spaces $\mathbb{E}_j(I)$ are defined as in \eqref{Eqn:MaxRegSpaces} above.
One quickly verifies that $\Phi$ is real--analytic on the domain specified 
and further, the Fr\'echet derivative of $\Phi$ at $r$ satisfies
\[
D_1 \Phi(r, (0, 0)) = \left( \partial_t - DG(r), \gamma \right) \in \mathcal{L}_{isom}(\mathbb{E}_1(I), \mathbb{E}_0(I) \times E_{\mu})
\]
by way of maximal regularity, c.f. \cite[Theorem 6.1]{CS01}.
The result now follows by an application of the implicit function theorem,
c.f. \cite[Theorem 2.4]{LS12}.
\end{proof}

%%%%%%%%%%%%%%%%%%%%%%%%%%%%%%%%%%%%%%%%%%%%%%%%%%%%%%%%%%%%%%%%%%%%%%%%%%%%%%%%%%%%%%%%%%%%%%%%%%%%%%%%%%%%%%%%%%
\section{General Properties of \eqref{AMC}}\label{Sec:Equilibria}
%%%%%%%%%%%%%%%%%%%%%%%%%%%%%%%%%%%%%%%%%%%%%%%%%%%%%%%%%%%%%%%%%%%%%%%%%%%%%%%%%%%%%%%%%%%%%%%%%%%%%%%%%%%%%%%%%%

%With well--posedness of \eqref{AMC} established above, 
We move on to investigate geometric properties of solutions. 
To begin, we verify that two important features of the 
averaged mean curvature flow hold for solutions of \eqref{AMC}, 
namely volume--preservation and surface--area reduction.
Following these verifications, we characterize the collection
of equilibria to \eqref{AMC}, which will be the central focus 
of the remainder of the article.

%%%%%%%%%%%%%%%%%%%%%%%%%%%%%%%%%%%%%%%%%%%%%%%%%%%%%%%%%%%%%%%%%%%%%%%%%%%%%%%%%%%%%%%%%%%%%%%%%%%%%%%%%%%%%%%%%%
\subsection{Volume--Preservation and Surface--Area Reduction}
%%%%%%%%%%%%%%%%%%%%%%%%%%%%%%%%%%%%%%%%%%%%%%%%%%%%%%%%%%%%%%%%%%%%%%%%%%%%%%%%%%%%%%%%%%%%%%%%%%%%%%%%%%%%%%%%%%

Suppose that $r \in C^1_{1 - \mu}(J(r_0), E_0) \cap C_{1 - \mu}(J(r_0), E_1)$ 
is a solution to \eqref{AMC} for some initial value $r_0 \in V_{\mu}$, $\mu \in [1/2, 1]$.
Looking at the evolution of the surface--area functional $S$, making
judiscious use of integration by parts and periodicity to cancel boundary terms,
we see that
\begin{align*}
\frac{d}{dt} S(r) &= \int_{\Tone} \left( \sqrt{1 + r_x^2(x)} + \frac{r(x) r_x(x)}{\sqrt{1 + r_x^2(x)}} \partial_x \right) G(r) \; dx\\
								%	&= \int_{\Tone} \left[ \sqrt{1 + r_x^2(x)} - \frac{r_x^2(x)}{\sqrt{1 + r_x^s(x)}} - \frac{r(x) r_{xx}(x)}{\sqrt{1 + r_x^2(x)}} + \frac{r(x) r_x^2(x) r_{xx}(x)}{(1 + r_x^2(x))^{3/2}} \right] G(r) dx\\
									%&= \int_{\Tone} \left( 1 + r_x^2(x) - r_x^2(x) - r(x) r_{xx}(x) + \frac{r(x) r_x^2(x) r_{xx}(x)}{1 + r_x^2(x)} \right) \Big[ h(r) - \cH(r)(x) \Big] dx\\
								%	&= \int_{\Tone} \left[ \frac{1}{r(x) \sqrt{1 + r_x^2(x)}} - \frac{r_{xx}(x) (1 + r_x^2(x)) - r_x^2(x) r_{xx}(x)}{(1 + r_x^2(x))^{3/2}} \right] (h(r) - \cH(r)(x)) r(x) \sqrt{1 + r_x^2(x)} dx\\
									&= \int_{\Tone} \cH(r)(x) \Big[ h(r) - \cH(r)(x) \Big] r(x) \sqrt{1 + r_x^2(x)} \; dx\\
									&= -\int_{\Tone} \Big[ h(r) - \cH(r)(x) \Big]^2 r(x) \sqrt{1 + r_x^2(x)} \; dx,
\end{align*}
where the last equation follows from 
\begin{align*}
\int_{\Tone} &h(r) \Big[ h(r) - \cH(r)(x) \Big] r(x) \sqrt{1 + r_x^2(x)} \; dx\\ 
									&= \Big[ h(r) \Big]^2 \left( \int_{\Tone} r(x) \sqrt{1 + r_x^2(x)} \; dx \right) - h(r) \left( \int_{\Tone} \cH(r)(x) r(x) \sqrt{1 + r_x^2(x)} \; dx \right)\\
									&= S(r) \Big[ h(r) \Big]^2 - S(r) \Big[ h(r) \Big]^2 = 0.
\end{align*}
Thus, it follows that the surface area of the induced surfaces of
revolution $\Gamma(r(t))$ is decreasing in $t$. In fact, 
we can see that the surface area is constant if and only if the 
mean curvature is constant, which coincides with an equilibrium surface.
Meanwhile, looking at the evolution of the volume functional
\[
F(r) := \int_{\Tone} r^2(x) \; dx,
\]
we see that
\begin{align*}\label{Eqn:VolumePreserved}
\frac{d}{dt} F(r) = 2 \int_{\Tone} r(x) G(r)(x) \; dx = 2 \int_{\Tone} r(x) \sqrt{1 + r_x^2(x)} (h(r) - \cH(r)(x)) \; dx = 0
									%&= 2 \frac{\int_{\Tone} r(x) \sqrt{1 + r_x^2(x)} dx}{\int_{\Tone} r(y) \sqrt{1 + r_x^2(y)} dy} 
									%\int_{\Tone} \cH(r)(y) r(y) \sqrt{1 + r_x^2(y)} dy \\ & \quad- 2 \int_{\Tone} r(x) \sqrt{1 + r_x^2(x)} \cH(r)(x) dx = 0.
\end{align*}
Therefore, taking into account the regularity of the operator $F$ in
the topology of $E_{\mu}$, it follows that $F(r(t)) = F(r_0)$ for
all $t \in J(r_0)$.

%%%%%%%%%%%%%%%%%%%%%%%%%%%%%%%%%%%%%%%%%%%%%%%%%%%%%%%%%%%%%%%%%%%%%%%%%%%%%%%%%%%%%%%%%%%%%%%%%%%%%%%%%%%%%%%%%%
\subsection{Characterization of Equilibria}
%%%%%%%%%%%%%%%%%%%%%%%%%%%%%%%%%%%%%%%%%%%%%%%%%%%%%%%%%%%%%%%%%%%%%%%%%%%%%%%%%%%%%%%%%%%%%%%%%%%%%%%%%%%%%%%%%%

Considering the equilibria, i.e. the steady states, of the problem \eqref{AMC},
it follows immediately that $\brho \in E_{\mu}$ is an equilibrium if
and only if $G(\brho) = 0$. Moreover, from the structure of the equation
we conclude $G(\brho) = 0$ if and only if $\cH(\brho) \equiv h(\brho),$ 
i.e. the mean curvature of the surface of revolution $\Gamma(\brho)$
must be constant on $\Tone$. 
Using results of Delaunay \cite{DE41} and Kenmotsu \cite{KEN80} regarding
surfaces of revolution with constant mean curvature, 
we characterize all equilibria of \eqref{AMC} as follows.

\begin{thm}[Characterization of Equilibria]
Fix $\alpha \in (0,1)$. Then $\brho \in h^{2 + \alpha}(\Tone)$ is an equilibrium
of the flow \eqref{AMC} if and only if $\brho$ is a $2 \pi$--periodic undulary curve.
Moreover, up to a translation along the axis of rotation, $\brho$ can be expressed 
explicitly by the parametric equation (with respect to the arclength parameter s),
\[
R(s; \mathcal{H}, B) := \Bigg( \int_{\pi / 2 \mathcal{H}}^s \frac{1 + B \sin ( \mathcal{H} t)}{\sqrt{1 + B^2 + 2 B \sin ( \mathcal{H} t)}} \; dt \, ,
\frac{\sqrt{1 + B^2 + 2 B \sin ( \mathcal{H} s)}}{ \, |\mathcal{H}|} \Bigg),
\]
for constants $\cH > 0$ and $B \in (-1, 1)$ which satisfy the relationship
\begin{equation}\label{Eqn:HandBRelation}
\frac{\pi \, \mathcal{H}}{k} = \int_{\pi/2}^{3 \pi/2} \frac{1 + B \sin t}{\sqrt{1 + B^2 + 2B \sin t}} \; dt \,, \qquad \text{for some} \quad k \in \mathbb{N}.
\end{equation}
%some $k \in \mathbb{N}$.
\end{thm}

\begin{proof}
Delaunay \cite{DE41} proved that every closed surface of 
revolution in $\mathbb{R}^3$ with constant mean 
curvature is either a sphere, a catenoid, a nodoid or an unduloid. Looking at the 
associated profile curves, the only surfaces which fit into our 
current setting (taking into account periodicity, embeddedness and 
regularity constraints) are the family of $2 \pi$--periodic unduloids, 
for which Kenmotsu \cite{KEN80} derived an explicit 
parametrization of profile curves. 
The parametrization presented herein is a
slightly modified version of Kenmotsu's formula, where we have adjusted the graph to
be symmetric about the $y$--axis and to exhibit $\frac{2 \pi}{k}$--periodicity in the 
$x$--variable, for a given value $k \in \mathbb{N}$.
%(in contrast to inherent periodicity in the arclength variable $s$).
\end{proof}

To stimulate the reader's curiosity, we include here some graphs of undulary curves
$\brho$ which, by the previous theorem, are the only equilibria of \eqref{AMC} 
(up to shifts along the $x$--axis).

%%%%%%%%%%%%%%%%%%%%%%%%%%%%%%%% FIGURES %%%%%%%%%%%%%%%%%%%%%%%%%%%%%%%%%%%%%%%%%%%%%%%%%%%%%%%%%%%%%%%%%%%%%%%%%%

\begin{figure}[ht]
\centering
\mbox{\subfigure{\includegraphics[width=2.4in,height=1.5in,clip=true,trim=.5in 2.5in .5in 2.5in]{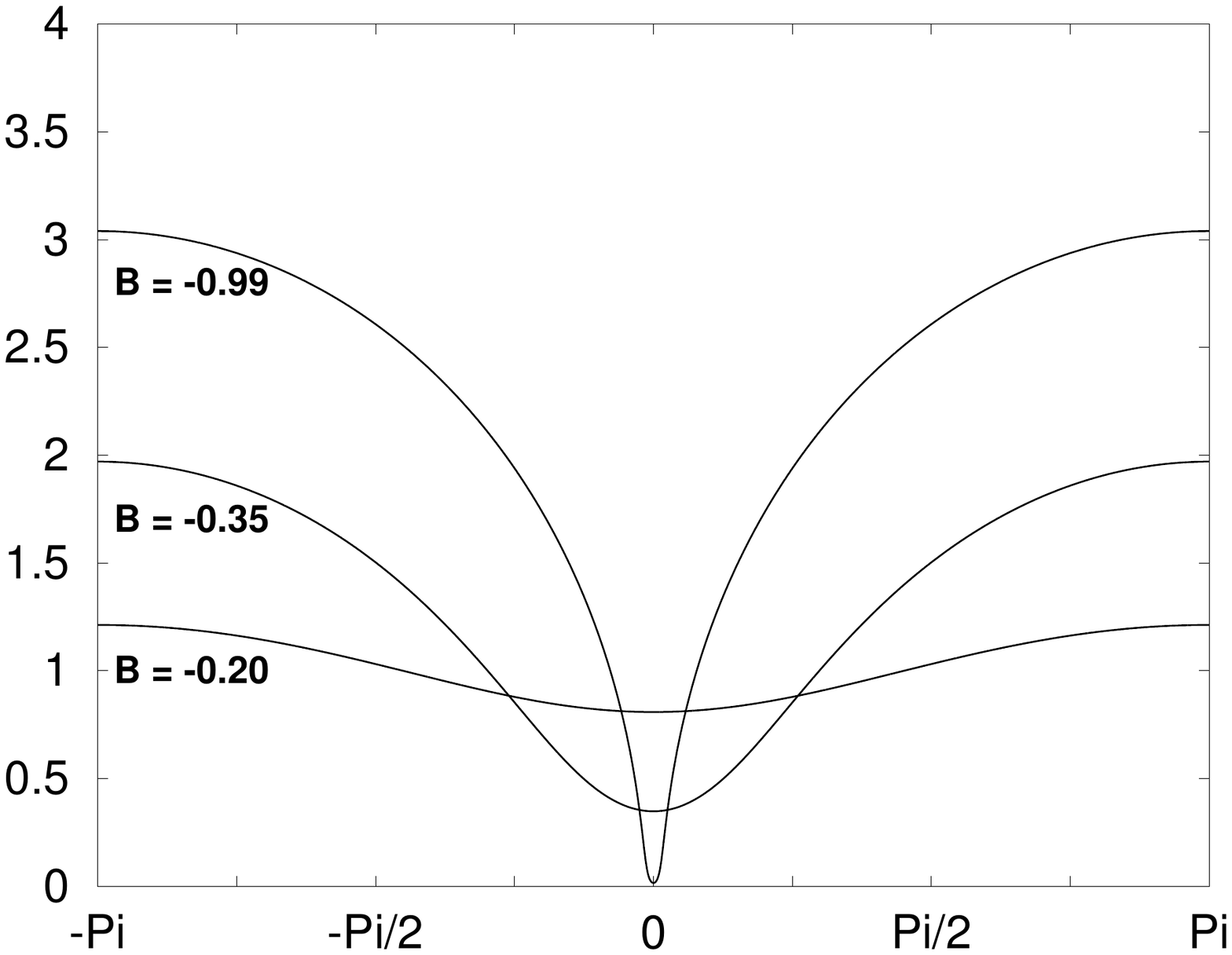}}\quad
\subfigure{\includegraphics[width=2.4in,height=1.5in,clip=true,trim=.5in 2.5in .5in 2.5in]{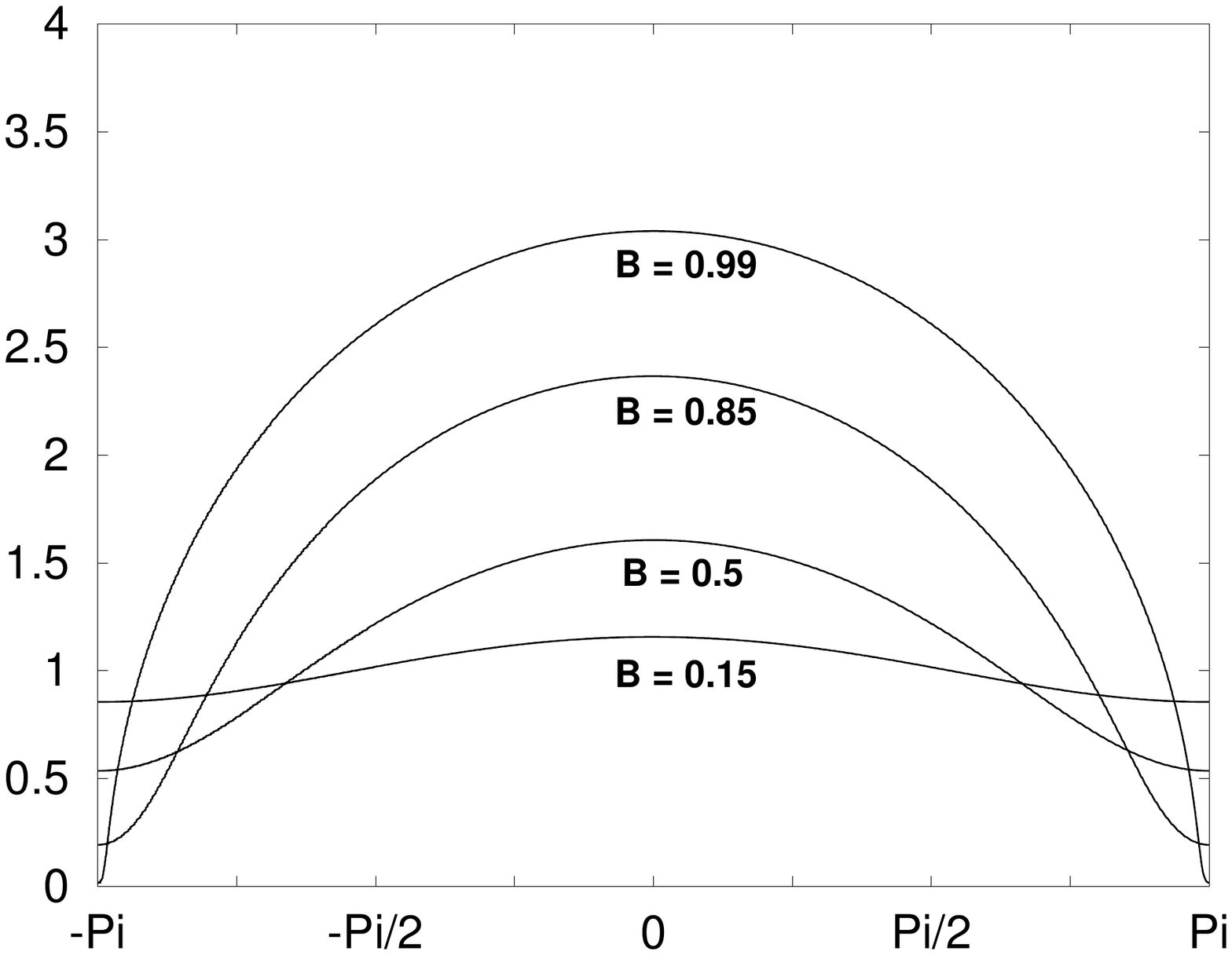} }}
\caption[Undulary Curves]{Families of $2 \pi$ periodic undulary curves (with $k = 1$) with values between $B = -.99$
and $B = 0.99$, as indicated.} \label{Fig:Unduloids}
\end{figure}

%\begin{figure}[ht]
%\centering
%\mbox{\subfigure{\includegraphics[width=2.4in,height=1.5in,clip=true,trim=.5in 2.5in .5in 2.5in]{FigNegBpiUnduloids.pdf}}\quad
%\subfigure{\includegraphics[width=2.4in,height=1.5in,clip=true,trim=.5in 2.5in .5in 2.5in]{FigpiUnduloids.pdf} }}
%\caption[Undulary Curves]{Families of $\pi$ periodic undulary curves--$k = 2$--with selected parameter values from $B = -.99$
%to $B = 0.99$, as indicated.} \label{Fig:Unduloids}
%\end{figure}

%%%%%%%%%%%%%%%%%%%%%%%%%%%%%%%% FIGURES %%%%%%%%%%%%%%%%%%%%%%%%%%%%%%%%%%%%%%%%%%%%%%%%%%%%%%%%%%%%%%%%%%%%%%%%%%

%%%%%%%%%%%%%%%%%%%%%%%%%%%%%%%%%%%%%%%%%%%%%%%%%%%%%%%%%%%%%%%%%%%%%%%%%%%%%%%%%%%%%%%%%%%%%%%%%%%%%%%%%%%%%%%%%%%
\section{Stability Behavior of Cylinders}\label{Sec:Stability}
%%%%%%%%%%%%%%%%%%%%%%%%%%%%%%%%%%%%%%%%%%%%%%%%%%%%%%%%%%%%%%%%%%%%%%%%%%%%%%%%%%%%%%%%%%%%%%%%%%%%%%%%%%%%%%%%%%%

For the remainder of the paper, we equate the constant function $\rs(x) \equiv \rs > 0$ with the 
cylinder $\Gamma(\rs)$ with radius $\rs$ and length $2 \pi$. The stability of cylinders under the
averaged mean curvature flow is a particularly vibrant question, in rotationally symmetric settings. 
Considering $n$--dimensional smooth surfaces of revolution with Neumann boundary conditions, 
Athanassenas \cite{ATH97} showed that initial surfaces satisfying a particular perimetric--type inequality, 
comparing the surface area, enclosed volume, and length of the interval of revolution,
will have global solutions which eventually converge to an $n$--dimensional cylinder.
More recent results by Hartley \cite{HAR12} indicate that one may be able to generalize this
convergence result to rough initial data without requiring rotational symmetry of initial data,
though we again see a perimetric--type inequality in his results.
Athanassenas notes in \cite{ATH97} that the volume constraint imposed by the 
perimetric--type inequality eliminates the possibility of unduloid equilibria,
in the absence of which one can concludes convergence to cylinders. 
However, in the current paper, we are interested precisely in the 
family of unduloids and how they interact (bifurcate from) the family of cylinders, for 
which a finer analysis is necessary. 

In this section, we will demonstrate that the family of cylinders with radius $\rs > 1$
is exponentially stable, c.f. Theorem~\ref{Thm:Stability}. 
%We note that this is
%an improvement upon the result of \cite{HAR12} in the current setting, from which one can
%only derive stability for radii $\rs > 2$. 
Moreover, we establish instability of cylinders
with radius $0 < \rs < 1$ and develop a setting within which dynamical systems
techniques become accessible. In the next section, we will show how this setting can
be applied to study the bifurcation which occurs at the critical radius $\rs = 1$, and
subsequent bifurcations at smaller radii, c.f. Theorem~\ref{Thm:Bifurcation}.

We first note that functions in the little--H\"older spaces, $r \in h^{\sigma}(\Tone)$,
obey the Fourier series representation (c.f. \cite{LeC11})
\[
r(x) = \sum_{k \in \mathbb{Z}} \hat{r}(k) \; e^{i k x}, \qquad x \in \Tone; \quad \text{where } \; \hat{r}(k) := \frac{1}{2 \pi} \int_{\Tone} r(x) e^{-ikx} \, dx.
\]
We then compute the linearization
\begin{align}\label{Eqn:DG}
DG(\rs) \rho &= - D \cH(\rs) \rho + \frac{1}{2 \pi} \int_{\Tone} [D \cH(\rs) \rho](x) \; dx, \qquad \rho \in h^{2 + \alpha}(\Tone)\\
&= \left( \rho \; \rs^{- 2} + \rho_{xx} \right) - \frac{1}{2 \pi} \int_{\Tone} (\rho(x) \; \rs^{- 2} + \rho_{xx}(x)) \; dx,
\end{align}
and realize it in the form
\[
\big[ DG(\rs) \rho \big] (x) = \sum_{k \in \mathbb{Z} \setminus \{ 0 \}} (\rs^{-2} - k^2) \; \hat{\rho}(k) \, e^{i k x}, \qquad \rho \in h^{2 + \alpha}(\Tone), \quad x \in \Tone.
\]
From this expression, it is easy to verify that the point spectrum of $DG(\rs)$ --
which must in fact coincide with the entire spectrum $\sigma( DG(\rs))$ due to compactness
of the embedding $h^{2 + \alpha}(\Tone) \hookrightarrow h^{\alpha}(\Tone)$ -- is
\begin{equation}\label{Eqn:SpectrumDG}
\sigma(DG(\rs)) = \{ 0 \} \cup \left\{ ( \rs^{-2} - k^2) : k \in \mathbb{Z} \setminus \{ 0 \} \right\}.
\end{equation}
Notice that, for $\rs > 1$, this spectrum is contained in the \emph{left--half} of the complex
plane, which we can take advantage of in order to establish stability of the equilibrium $\rs$. 
However, the presence of $0$ in the spectrum is troublesome, as it indicates the presence
of a center manifold which disturbs, but does not derail, our stability argument. 
To get around this hurdle, we proceed to reinterpret the operator $G,$ 
and subsequently $DG$, in a \emph{reduced} setting where this 
eigenvalue no longer shows up. The following technique takes advantage of the volume--preserving 
nature of \eqref{AMC} in order to investigate local properties of the problem, it is motivated 
by methods of Prokert \cite{PRO97} and Vondenhoff \cite{VON08} and employed by the author in \cite{LS12}.

%%%%%%%%%%%%%%%%%%%%%%%%%%%%%%%%%%%%%%%%%%%%%%%%%%%%%%%%%%%%%%%%%%%%%%%%%%%%%%%%%%%%%%%%%%%%%%%%%%%%%%%%%%%%%%%%%%%
\subsection{Zero--Mean Functions and the Reduced Problem}
%%%%%%%%%%%%%%%%%%%%%%%%%%%%%%%%%%%%%%%%%%%%%%%%%%%%%%%%%%%%%%%%%%%%%%%%%%%%%%%%%%%%%%%%%%%%%%%%%%%%%%%%%%%%%%%%%%%

For this subsection, we consider an arbitrary $\rs > 0$ and $\sigma \in \mathbb{R}_+ \setminus \mathbb{Z}$,
unless otherwise stated. We first introduce the mappings
\[
P_0 r := r - \frac{1}{2 \pi}\int_{\Tone} r(x) \; dx, \quad \text{and} \quad Q_0 := 1 - P_0
\]
which define projections on $h^{\sigma}(\Tone)$.
Denote by $h^{\sigma}_0(\Tone)$ the image $P_0 \big( h^{\sigma}(\Tone) \big),$ which 
coincides with the zero-mean functions on $\Tone$ in the regularity class
$h^{\sigma}(\Tone)$. We then have the topological decomposition
\[
h^{\sigma}(\Tone) = h^{\sigma}_0(\Tone) \oplus Q_0 \big( h^{\sigma}(\Tone) \big) 
\cong h^{\sigma}_0(\Tone) \oplus \mathbb{R} \,.
\]
In the sequel, we equate the constant function 
$[\eta(x) \equiv \eta] \in Q_0 \big( h^{\sigma}(\Tone) \big)$
with the value $\eta \in \mathbb{R},$ and we denote each simply as $\eta$. 

Consider the operator
\[
\Phi(r, \tilde{r}, \eta) := \Big( P_0 r - \tilde{r}, \; F(r) - F(\eta) \Big), \qquad \text{with} \quad F(r) := \int_{\Tone} r^2(x) \; dx,
\]
which is a real--analytic operator from 
$h^{\sigma}(\Tone) \times h^{\sigma}_0(\Tone) \times \mathbb{R}$ into $h^{\sigma}_0(\Tone) \times \mathbb{R}$.
Notice that $\Phi(\rs, 0, \rs) = (0, 0)$ and
\begin{equation}\label{Eqn:D1Phi}
\begin{split}
D_1 \Phi (\rs, 0, \rs) &= \Big( P_0 , \; 2 \int_{\Tone} \rs id_{h^{\sigma}(\Tone)} \; dx \Big)\\
&= \Big( P_0, \; 4 \pi \rs Q_0 \Big) \in \mathcal{L}_{isom}(h^{\sigma}(\Tone), h^{\sigma}_0(\Tone) \times \mathbb{R}),
\end{split}
\end{equation} 
i.e. the Fr\'echet derivative of $\Phi$ with respect to the first variable, at $(\rs, 0, \rs)$, is a 
bounded linear isomorphism. Hence, it follows from the implicit function theorem,
c.f. \cite[Theorem 15.3]{DE85}, that there exist neighborhoods 
\begin{equation}\label{Eqn:ImplicitNHoods}
(0, \rs) \in U = U_0 \times U_1 \subset h^{\sigma}_0(\Tone) \times \mathbb{R}, \qquad \rs \in U_{2} \subset h^{\sigma}(\Tone),
\end{equation}
and a $C^{\omega}$ function $\psi_{\star}: U \rightarrow U_{2}$ such that, for all $(r, \tilde{r}, \eta) \in U_{2} \times U$,
\[
\Phi(r, \tilde{r}, \eta) = ( 0, 0 ) \qquad \text{if and only if} \qquad r = \psi_{\star}(\tilde{r}, \eta).
\]

Note that the quantity $F(r)$ corresponds to the volume over $\Tone$ enclosed by the
surface $\Gamma(r)$ (modulo a factor of $2 \pi$), which is a preserved quantity for \eqref{AMC};
i.e. if $r(\cdot, r_0)$ is a solution to \eqref{AMC}, for an appropriate initial value 
$r_0 \in h^{2 \mu + \alpha}(\Tone)$,
it follows that $F(r(t)) = F(r_0)$ for all $t \in J(r_0)$.
Then, we can interpret $\psi_{\star}(\cdot, \eta)$ as a \emph{lifting}
operator which takes a zero--mean function $\tilde{r} \in h^{\sigma}_0(\Tone)$ to an associated 
function in the equivolume set
\[
\mathcal{M}^{\sigma}_{\eta} := \{ r \in h^{\sigma}(\Tone): F(r) = F(\eta) \},
\]
which contains all the profile functions corresponding to surfaces enclosing the same volume as
the cylinder $\eta$. In particular, if we choose $\eta$ so that $F(r_0) = F(\eta)$,
it follows that $r(t, r_0) \in \mathcal{M}^{\alpha}_{\eta}$ for all $t \in J(r_0)$.

\begin{rem}\label{Rem:PsiProperties}
For $\rs > 0$ fixed, we can immediately state the following properties of 
$\psi_{\star}$ which come from the definition and elucidate 
the relationship between $P_0$ and $\psi_{\star}$.
\begin{enumerate}

\item $P_0 \psi_{\star}(\tilde{r}, \eta) = \tilde{r}$ for all $(\tilde{r}, \eta) \in U$.

\item Given $r \in \psi_{\star}(U) \cap \mathcal{M}^{\sigma}_{\eta}$, it follows that $\psi_{\star}(P_0 r, \eta) = r \, .$ 

\item $\psi_{\star} (0, \eta) = \eta$, for $\eta \in U_1$.
This and the preceding remark follow from the fact that $F(\eta)$
is injective as a function on $(0, \infty)$.

\item It follows from the identity $\Phi( \psi_{\star}(\tilde{r}, \eta), \tilde{r}, \eta) = (0,0)$ and differentiating with
respect to $\tilde{r}$ that 
\[
D_1 \psi_{\star}(0,\eta)h = h, \qquad h \in h^{\sigma}_0(\Tone), \; \eta \in U_1.
\]

\item Note $\psi_{\star}(U_0, \eta) \subset \mathcal{M}^{\sigma}_{\eta}$ for $\eta \in U_1$
and we have the representation
\[
\psi_{\star}(\tilde{r}, \eta) = \Big(P_0 + Q_0 \Big) \psi_{\star}(\tilde{r}, \eta) = \tilde{r} + \frac{1}{2 \pi} \int_{\Tone} \psi_{\star}(\tilde{r}, \eta)(x) dx, \qquad \tr \in U_0.
\]
Hence, we see that $\mathcal{M}^{\sigma}_{\eta} \cap U_2$ is the graph of a $\mathbb{R}$--valued function
over $h^{\sigma}_0(\Tone)$ which is therefore a Banach manifold, with real--analytic 
local parametrization $\psi_{\star}(\cdot, \eta)$.

\item A priori, $\psi_{\star}(\cdot, \eta)$ depends upon the parameter $\sigma$, however it follows from the 
preceding representation that
\[
\psi_{\star}(\cdot, \eta): U_0 \cap h^{\tilde{\sigma}}_0(\Tone) \rightarrow h^{\tilde{\sigma}}(\Tone), \quad \tilde{\sigma} \in \mathbb{R}_+ \setminus \mathbb{Z},
\]
so that $\psi_{\star}$ preserves the spacial regularity of functions regardless of the regularity parameter $\sigma$
with which $\psi_{\star}$ was constructed. However, notice that the neighborhood $U_0$ will remain intrinsically linked
with the parameter which was used to construct $\psi_{\star}$.

\end{enumerate}
\end{rem}

Fix $\rs > 0$ and define the \emph{reduced} governing operator
\begin{equation}\label{Eqn:mGDefined}
\mathcal{G}(\tr, \eta) := P_0 G(\psi_{\star}(\tr, \eta))
\end{equation}
which acts on $(\tr, \eta)$ in a neighborhood of $U$, as constructed for 
$\psi_{\star}$ above.
We need to be careful with the regularity of the zero--mean functions $\tr$ that we
plug into $\mathcal{G}$, so we will assume throughout that $\alpha \in (0,1)$ is 
given, and then we define the lifting $\psi_{\star}$ and the related neighborhood
$U = U_0 \times U_1$ within the setting of $h^{\alpha}_0(\Tone).$ 
Then, we take $\mathcal{G}(\cdot, \eta)$ as acting
on the functions $\tr \in U_0 \cap h^{2 + \alpha}_0(\Tone)$, $\eta \in U_1$.
With this reduced operator, we defined the $\eta$--dependent reduced problem 
\begin{equation}\label{ReducedAMC}
\begin{cases} \tr_{t}(t,x) = \mathcal{G}(\tr(t,x), \eta), &\text{$t > 0, \, x \in \Tone,$}\\
\tr(0,x) = (P_0 r_0)(x), &\text{$x \in \Tone$.} \end{cases}
\end{equation}
Theoretically, we can choose the parameter $\eta \in U_1$ arbitrarily, 
however, in practice, we will choose $\eta$ for which $F(\eta) = F(r_0)$,
which essentially gives us the freedom to consider non--volume preserving
perturbations $r_0$ of the cylinder $\rs$ in our stability analysis.

Fix $\alpha \in (0,1)$ and we denote the spaces 
\[
F_0 := h^{\alpha}_0(\Tone), \quad F_1 := h^{2 + \alpha}_0(\Tone), \quad \text{and} \quad
F_{\mu} := (F_0, F_1)_{\mu, \infty}^0, \quad \mu \in (0,1),
\]
so that $F_{\mu} = P_0 E_{\mu}$ for $\mu \in [0,1]$. Moreover, for
$\mu \in (0, 1]$ and closed intervals $J \subseteq \mathbb{R}_+$ with $0 \in J$, define the spaces
\begin{align*}
\mathbb{E}_0(J) &:= BU\!C_{1 - \mu}(J, E_0),\\
\mathbb{E}_1(J) &:= BU\!C_{1 - \mu}^1(J, E_0) \cap BU\!C_{1 - \mu}(J, E_1),
\end{align*}
and 
\begin{align*}
\mathbb{F}_{0}(J) &:= BU\!C_{1 - \mu}(J, F_0),\\
\mathbb{F}_{1}(J) &:= BU\!C_{1 - \mu}^1(J, F_0) \cap BU\!C_{1 - \mu}(J, F_1),
\end{align*}
within which we will discuss solutions to \eqref{AMC} and 
the reduced problem \eqref{ReducedAMC}, respectively. We can immediately
see how the lifting operator $\psi_{\star}$ connects these two problems.

\begin{lemma}\label{Lem:PsiLifts}
Fix $\rs > 0$, $\eta \in U_1$ and $J := [0, T]$ for $T > 0$. Then
\begin{equation}\label{Eqn:PsiMaps}
\psi_{\star}(\cdot, \eta) : \mathbb{F}_1(J) \cap C(J, U_0) \longrightarrow \mathbb{E}_1(J), \qquad \text{with} \quad \psi_{\star}(\tr, \eta)(t) := \psi_{\star}(\tr(t), \eta).
\end{equation}
Moreover, if $\tr_0 := P_0 r_0 \in F_{\mu}$ and $\tr = \tr(\cdot, \tr_0) \in \mathbb{F}_1(J) \cap C(J, U_0)$ 
is a solution to \eqref{ReducedAMC}, for some $\mu \in [1/2, 1]$, then $r := \psi_{\star}(\tr, \eta)$ 
is the unique solution on the interval $J$ to \eqref{AMC}, with initial data 
$r_0 := \psi_{\star}(\tr_0, \eta) \in E_{\mu}$. 
\end{lemma}

\begin{proof}
Because of the local nature of the operator $\psi_{\star}$, we can assume, without 
loss of generality, that the neighborhood $U = U_0 \times U_1$ is chosen sufficiently
small to ensure that $\psi_{\star}$ is in the regularity class $C^{\omega}$ and the bounds
\begin{equation}\label{Eqn:PsiBound}
\| \psi(\tr, \eta) \|_{E_0} \leq N \quad \text{and} \quad \| D_1 \psi(\tr, \eta) \|_{\mathcal{L}(F_0, E_0)} \leq N 
\end{equation}
hold uniformly for $(\tr, \eta) \in U.$
Further, the embeddings
\begin{equation}\label{Eqn:Embedding}
\mathbb{F}_1(J) \hookrightarrow BU\!C(J, F_{\mu}) \hookrightarrow BU\!C(J, F_0), \qquad \mu \in [1/2, 1],
\end{equation}
follow from \cite[Theorem III.2.3.3]{AM95} and the continuous embedding of little-H\"older 
spaces, respectively. It is then straightforward, utilizing the representation for $\psi_{\star}$ 
given in Remarks~\ref{Rem:PsiProperties}(e) and these bounds, to verify \eqref{Eqn:PsiMaps}. 

Meanwhile, by \eqref{Eqn:Embedding} we have
$r_0 := \psi_{\star}(\tr_0, \eta) \in V_{\mu}$, and so it follows from
Theorem~\ref{Thm:WellPosedness} that there exists a unique maximal solution
\[
r(\cdot, r_0) \in C_{1 - \mu}^1(J(r_0), E_0) \cap C_{1 - \mu}(J(r_0), E_1)
\]
to \eqref{AMC}. Define $\rho(\cdot) := \psi(\tr(\cdot),\eta)$ as indicated 
and it suffices to show that $\rho_t(t) = G(\rho(t))$ for $t \in \dot{J} := (0, T]$, 
since this will imply that $\rho(t) = r(t, r_0)$ by uniqueness 
and maximality of the solution $r(\cdot, r_0)$. So,
let $t \in \dot{J}$ and consider the auxiliary problem  
\[
\begin{cases} \dot{\gamma}(\tau) = G(\gamma(\tau)), &\text{for $\tau \in [0, \varepsilon],$}\\
\gamma(0) = \rho(t), \end{cases}
\] 
which has a unique solution $\gamma \in C^1([0, \varepsilon], E_0) \cap C([0, \varepsilon], E_1)$
by Theorem~\ref{Thm:WellPosedness}, provided we choose $\varepsilon > 0$ sufficiently small
for the particular value $\rho(t) \in E_1$.  
By Remarks~\ref{Rem:PsiProperties} and preservation of volume for solutions of \eqref{AMC},
we have the representation $\gamma(\tau) = \psi_{\star}( P_0 \gamma(\tau), \eta)$,
$\tau \in [0, \varepsilon]$. We then conclude the proof by computing
\begin{align}\label{Eqn:GTangent}
G(\rho(t)) &= \dot{\gamma}(0) = \partial_{\tau} \left( \psi_{\star}(P_0 \gamma(\tau), \eta) \right) \Big|_{\tau = 0} 
= D_1 \psi_{\star}(P_0 \gamma(0), \eta) P_0 \dot{\gamma}(0) \nonumber\\
&= D_1 \psi_{\star}(P_0 \rho(t), \eta) P_0 G( \rho(t)) = D_1 \psi_{\star}( \tr(t), \eta) \mathcal{G}(\tr(t), \eta)\\
&= \partial_t \left( \psi_{\star}( \tr(t), \eta) \right) = \rho_t(t). \nonumber \qedhere
\end{align}
\end{proof}

Considering the stability of $\rs$, we compute the linearization
\[
D_1 \mathcal{G}(\tr, \eta) = P_0 DG(\psi_{\star}(\tr, \eta)) D_1 \psi_{\star}(\tr, \eta), \qquad (\tr, \eta) \in U,
\]
which simplifies at $\tr = 0$, using \eqref{Eqn:DG} and Remarks~\ref{Rem:PsiProperties}(d),
\begin{equation}\label{Eqn:DmGat0}
D_1 \mathcal{G}(0, \eta) \trho = P_0 \Big( DG(\eta) \Big) D_1 \psi(0, \eta) \trho = - P_0 \Big( P_0 D\cH(\eta) \Big) \trho = - P_0 D \cH(\eta) \trho
\end{equation}
for $\trho \in h^{2 + \alpha}_0(\Tone)$.
Recalling \eqref{Eqn:SpectrumDG}, and the reduced domain of definition
for $\mathcal{G}$, we conclude
\begin{equation}\label{Eqn:DmGSpectrum}
\sigma(D_1 \mathcal{G}(0, \eta)) = \left\{ (\eta^{-2} - k^2) : k \in \mathbb{Z} \setminus \{ 0 \} \right\}. 
\end{equation}

%%%%%%%%%%%%%%%%%%%%%%%%%%%%%%%%%%%%%%%%%%%%%%%%%%%%%%%%%%%%%%%%%%%%%%%%%%%%%%%%%%%%%%%%%%%%%%%%%%%%%%%%%%%%%%%%%%%
\subsection{Stability of Cylinders with $\rs > 1$}
%%%%%%%%%%%%%%%%%%%%%%%%%%%%%%%%%%%%%%%%%%%%%%%%%%%%%%%%%%%%%%%%%%%%%%%%%%%%%%%%%%%%%%%%%%%%%%%%%%%%%%%%%%%%%%%%%%%

Define the exponentially weighted maximal regularity spaces
\[
\mathbb{F}_j(\mathbb{R}_+, \omega) := \Big\{ f : (0, \infty) \rightarrow F_0 
\; \Big| \; [t \mapsto e^{\omega t} f(t)] \in \mathbb{F}_j(\mathbb{R}_+) \Big\}, \qquad \omega \in \mathbb{R}, \, j = 0,1,
\]
which are Banach spaces when equipped with norms
$\| u \|_{\mathbb{F}_j(\mathbb{R}_+, \omega)} := \| e^{\omega t} u \|_{\mathbb{F}_j(\mathbb{R}_+)}.$
We use these spaces in order to show the following stability result, 
which essentially tells us that $h^{2 \mu + \alpha}$--small perturbations of the cylinder $\rs > 1$ 
will have global solutions which converge exponentially fast to a cylinder $\eta$, which
is close to $\rs$.

\begin{thm}\label{Thm:Stability}
Fix $\alpha \in (0,1), \; \mu \in [1/2, 1]$ so that $2 \mu + \alpha \notin \mathbb{Z}$, and let $\rs > 1$. 
There exist positive constants $\varepsilon = \varepsilon(\rs), \; \delta = \delta(\rs)$ and 
$\omega = \omega(\rs, \delta),$ such that problem \eqref{AMC} with initial data 
$r_0 \in \mathbb{B}_{h^{2 \mu + \alpha}}(\rs, \varepsilon)$ has a unique global solution
\[
r(\cdot, r_0) \in C^1_{1 - \mu} (\mathbb{R}_+, h^{\alpha}(\Tone)) \cap C_{1 - \mu}(\mathbb{R}_+, h^{2 + \alpha}(\Tone)),
\]
and there exist $\eta = \eta(r_0) \in (\rs - \delta, \rs + \delta)$ and $M = M(\alpha) > 0$ for which the bound
\[
t^{1 - \mu} \| r(t, r_0) - \eta \|_{h^{2 + \alpha}} + \| r(t, r_0) - \eta \|_{h^{2 \mu + \alpha}} \le M e^{-\omega t} \| r_0 - \rs \|_{h^{2 \mu + \alpha}}
\]
holds uniformly for $t \ge 0$.
\end{thm}

\begin{proof}
We demonstrate this result by first showing that $0$ is exponentially stable in 
the reduced problem \eqref{ReducedAMC}, 
and then we lift solutions using $\psi_{\star}$ and show that 
exponential convergence is preserved in the lifting process.

{\bf (i)} Fix $\delta \in (0, \rs - 1)$. Notice that the linearization of $\mathcal{G}(0,\eta)$
has the structure
\[
D_1 \mathcal{G}(0, \eta) \trho = P_0 \Big( h_{xx} + \eta^{-2} \, \trho \Big) =  \Big( h_{xx} + \eta^{-2} \, \trho \Big) - \frac{1}{2 \pi} \int_{\Tone} \Big( \trho_{xx} (x) + \eta^{-2} \, \trho(x) \Big) \, dx, 
\]
from which, similar to the argument presented in Theorem~\ref{Thm:WellPosedness}, 
we realize $D_1 \mathcal{G}(0, \eta)$ as a lower order perturbation
of the second order derivative operator $\partial_x^2$. Hence, it follows that
$D_1 \mathcal{G} (0, \eta) \in \mathcal{H}(E_1, E_0)$; the class of infinitesimal generators of 
analytic semigroups on $E_0$ with domain $E_1$. Moreover, it follows from \eqref{Eqn:DmGSpectrum}
that the spectral type 
\[
type(D_1 \mathcal{G}(0, \eta)) < \frac{1 - (\rs - \delta)^{2}}{(\rs - \delta)^2} < 0, \qquad \text{for all} \quad \eta \in (\rs - \delta, \rs + \delta).
\]
Hence, if we choose $\omega \in \left( 0, \frac{(\rs - \delta)^2 - 1}{(\rs - \delta)^2} \right)$, 
then it follows by \cite[Theorem III.3.4.1 and Remarks 3.4.2(b)]{AM95} that the exponentially weighted spaces ,
\[
\Big( \mathbb{F}_0(\mathbb{R}_+, \omega), \; \mathbb{F}_1(\mathbb{R}_+, \omega) \Big),
\]
are a pair of maximal regularity for $D_1 \mathcal{G}(0, \eta)$, for all $\eta \in (\rs - \delta, \rs + \delta)$.

Define the operator
\[
\mathcal{K}(\tr, \tr_0, \eta) := \Big( \partial_t \tr - \mathcal{G}(\tr, \eta), \; \gamma \tr - \tr_0  \Big),
\]
acting on $\mathbb{U} := \Big( \mathbb{F}_1(\mathbb{R}_+, \omega) \cap C(\mathbb{R}_+, U_0) \Big) \times \Big(U_0 \cap F_{\mu} \Big) \times U_1 \, ,$
which is an open subset of $\mathbb{F}_1(\mathbb{R}_+, \omega) \times F_{\mu} \times \mathbb{R}$. 
It follows that
\[
\mathcal{K} \in C^{\omega} \Big( \mathbb{U}, \mathbb{F}_0(\mathbb{R}_+, \omega) \times F_{\mu} \Big),
\]
by analyzing the mapping properties of the individual operators $\gamma, \, \partial_t$ and $\mathcal{G}$
on their associated domains of definition. The fact that $\mathcal{G}(\cdot, \eta)$ maps 
$\Big( \mathbb{F}_1(\mathbb{R}_+, \omega) \cap C(\mathbb{R}_+, U_0) \Big)$ into 
$\mathbb{F}_0(\mathbb{R}_+, \omega)$ follows by utilizing the representation \eqref{Eqn:QuasiStructure}, 
and bounding individual terms of the resulting expression using the embeddings \eqref{Eqn:Embedding} and 
the exponential boundedness of functions in $\mathbb{F}_1(\mathbb{R}_+, \omega)$.

Meanwhile, notice that $\mathcal{K}(0,0,\eta) = (0, 0)$ and
\[ 
D_1 \mathcal{K}(0,0,\eta) = \Big( \partial_t - D_1 \mathcal{G}(0,\eta), \gamma  \Big) \in \mathcal{L}_{isom}
\Big(\mathbb{F}_1(\mathbb{R}_+, \omega), \mathbb{F}_0(\mathbb{R}_+, \omega) 
\times F_{\mu} \Big),
\]
for $\eta \in (\rs - \delta, \rs + \delta)$. Thus, by the implicit function 
theorem there exists an open neighborhood 
$(0, \rs) \in \tilde{U} \subset F_{\mu} \times \mathbb{R}$ and a $C^{\omega}$ mapping 
$\kappa : \tilde{U} \rightarrow \mathbb{F}_1(\mathbb{R}_+, \omega)$ such that 
\[
\mathcal{K}(\kappa(\tr_0, \eta), \tr_0, \eta) = (0, 0) \quad \text{for all} \quad (\tr_0, \eta) \in \tilde{U}.
\]
In particular, $\kappa(\tr_0, \eta)$ is a global solution to \eqref{ReducedAMC} with parameter $\eta$ and 
initial data $\tr_0 \in F_{\mu}$, which converges exponentially fast to 0. 
Without loss of generality, we assume $\tilde{U} \subset U_1 \times (\rs - \delta, \rs + \delta)$.

{\bf (ii)} Choose $\varepsilon > 0$ so that, for every $r_0 \in \mathbb{B}_{E_{\mu}}(\rs, \varepsilon)$,
there exists $(\tr, \eta) \in \tilde{U}$ for which $\Phi(r_0, \tr, \eta) = (0,0)$. 
Let $r_0 \in \mathbb{B}_{E_{\mu}}(\rs, \varepsilon)$, fix $\eta \in (\rs - \delta, \rs + \delta)$
so that $F(r_0) = F(\eta)$ and define
\[
r(\cdot) := \psi_{\star}(\kappa(P_0 r_0, \eta), \eta).
\]
Using Lemma~\ref{Lem:PsiLifts}, one verifies that 
$r \in C^1_{1 - \mu} (\mathbb{R}_+, h^{\alpha}(\Tone)) \cap C_{1 - \mu}(\mathbb{R}_+, h^{2 + \alpha}(\Tone))$ 
is the unique global solution to \eqref{AMC} with initial data $r_0$. 
Hence, it remains to show exponential convergence of $r$ to the cylinder $\eta$.
Noting that $\kappa(0, \eta) \equiv 0$, and utilizing Remarks~\ref{Rem:PsiProperties} and 
the mean value theorem, we compute
\begin{align*}
r(t) - & \eta = (P_0 + Q_0) \Big( \psi_{\star}(\kappa(P_0 r_0, \eta)(t), \eta) - \psi_{\star}(\kappa(0, \eta)(t), \eta) \Big)\\
%&= \Big( P_0 + (1 - P_0) \Big) \Big( \psi(\kappa(P_0r_0, \eta)(t), \eta) - \psi(\kappa(0, \eta)(t), \eta) \Big)\\
&= \kappa(P_0 r_0, \eta)(t) + Q_0 \Big( \psi_{\star}(\kappa(P_0 r_0, \eta)(t), \eta) - \psi_{\star}(\kappa(0, \eta)(t), \eta) \Big)\\
&= \kappa(P_0 r_0, \eta)(t) + \frac{1}{2 \pi} \int_{\Tone} \int_0^1 D_1 \psi_{\star} \big( \tau \kappa(P_0 r_0, \eta)(t), \eta \big) \kappa(P_0 r_0, \eta)(t, x) \; d\tau \, dx,
\end{align*}
for $t > 0$. We bound the terms
\[
e^{\omega t} t^{1 - \mu} \| \kappa(P_0 r_0, \eta)(t) \|_{F_0} \qquad \text{and} \qquad e^{\omega t} \| \kappa(P_0 r_0, \eta)(t) \|_{F_{\mu}}
\]
uniformly for $t \ge 0$ using the property $\kappa(P_0 r_0, \eta) \in \mathbb{F}_1(\mathbb{R}_+, \omega)$ 
and \eqref{Eqn:Embedding}, respectively. Meanwhile, bounding integral terms in the
$C(\Tone)$--topology and using the bounds \eqref{Eqn:PsiBound}, we get
%to conclude that 
\begin{equation}\label{Eqn:Bound2}
e^{\omega t} t^{1 - \mu} \| r(t) - \eta \|_{E_1} \leq \Big(1 + c_1 N \Big) \| \kappa(P_0 r_0, \eta) \|_{\mathbb{F}_1(\mathbb{R}_+, \omega)}, \qquad t \ge 0,
\end{equation} 
\begin{equation}\label{Eqn:Bound3}
e^{\omega t} \| r(t) - \eta \|_{E_{\mu}} \leq \Big( c_2 + c_3 N \Big) \| \kappa(P_0r_0, \eta) \|_{\mathbb{F}_1(\mathbb{R}_+, \omega)}, \qquad t \ge 0.
\end{equation}
Here the constant $c_1$ comes from the embedding 
$F_1 \hookrightarrow F_0$, and the constants $c_2$ and $c_3$ come from 
the embeddings \eqref{Eqn:Embedding}. Finally, by the regularity of 
$\kappa$, we may assume that $\tilde{U}$ was chosen sufficiently small to ensure that $D_1 \kappa$ is uniformly
bounded from $\tilde{U}$ into $\mathbb{F}_1(\mathbb{R}_+, \omega).$
Recalling that $\kappa(0, \eta) = 0$, we have
\begin{equation}\label{Eqn:Kappa}
\begin{split}
\| \kappa(P_0 r_0, \eta) \|_{\mathbb{F}_1(\mathbb{R}_+, \omega)} &\leq \int_0^1 \left\| D_1 \kappa(\tau P_0 r_0, \eta) P_0 r_0 \right\|_{\mathbb{F}_1(\mathbb{R}_+, \omega)} \; d \tau\\
&\leq \tilde{M} \| P_0 r_0 \|_{F_{\mu}} \leq M \|r_0 - \rs \|_{E_{\mu}},
\end{split}
\end{equation}
where $M := \| P_0 \| \sup_{(\tr, \eta) \in \tilde{U}} \|D_1 \kappa(\tr, \eta) \|_{\mathcal{L}(
F_{\mu}, \mathbb{F}_1(\mathbb{R}_+, \omega))}.$ 
The claim now follows from \eqref{Eqn:Kappa} and
the inequalities \eqref{Eqn:Bound2}--\eqref{Eqn:Bound3}. 
\end{proof}

%%%%%%%%%%%%%%%%%%%%%%%%%%%%%%%%%%%%%%%%%%%%%%%%%%%%%%%%%%%%%%%%%%%%%%%%%%%%%%%%%%%%%%%%%%%%%%%%%%%%%%%%%%%%%%%%%%%
\subsection{Instability of Cylinders with $0 < \rs < 1$}
%%%%%%%%%%%%%%%%%%%%%%%%%%%%%%%%%%%%%%%%%%%%%%%%%%%%%%%%%%%%%%%%%%%%%%%%%%%%%%%%%%%%%%%%%%%%%%%%%%%%%%%%%%%%%%%%%%%

\begin{thm}\label{Thm:Instability}
Let $\rs \in (0,1)$ and $\mu \in [1/2, 1]$ be fixed with $2 \mu + \alpha \notin \mathbb{Z}$. 
Then the equilibrium $\rs$ of \eqref{AMC}
is unstable in the topology of $h^{2 \mu + \alpha}(\Tone)$
for initial values in $h^{2 \mu + \alpha}(\Tone)$.
\end{thm}

\begin{proof}
{\bf (i)}
Let $\rs \in (0,1)$ be fixed, and let $L := DG(\rs)$.
We can restate the evolution equation \eqref{AMC} 
in the following equivalent form
\begin{equation}
\label{reduced-2}
\begin{cases}
\rho_t - L \rho = g(\rho), &\text{$t > 0$} \\ \rho(0) = r_0 - \rs,
\end{cases}
\end{equation}
where $g(\rho) := G(\rho + \rs) - L \rho$.
Using the quasilinear structure of $G$ 
it is not difficult to see that for every $\beta > 0$ there exists a 
number $\varepsilon_0 = \varepsilon_0(\beta) > 0$ so that
\begin{equation}
\label{estimate-g}
\|g(\rho)\|_{E_0} \leq \beta \| \rho \|_{E_1}, \qquad \text{for all} \quad 
\rho \in \mathbb{B}_{E_{\mu}}(0, \varepsilon_0) \cap E_1.
\end{equation}
It follows from \eqref{Eqn:SpectrumDG} that
\begin{equation*}
\sigma(L) \cap [{\rm Re} \, z > 0] \ne \emptyset,
\end{equation*}
and we may choose numbers $\omega, \gamma > 0$ such that
\begin{equation*}
[\omega - \gamma \le {\rm Re} \,z \le \omega + \gamma] \cap \sigma(L) = \emptyset
\quad \text{and} \quad
\sigma_+ := [{\rm Re} \, z > \omega + \gamma] \cap \sigma(L) \ne \emptyset \, .
\end{equation*}
%i.e. the strip $[\omega - \gamma \le {\rm Re} \, z \le \omega + \gamma]$ does not intersect
%$\sigma(L)$ and there is at least one point of $\sigma(L)$ to the right  of
%the line $[{\rm Re} \, z = \omega + \gamma]$.

Define $P_+$ to be the spectral projection, in $E_0$, 
with respect to the spectral set $\sigma_+$,
and let $P_- := 1 - P_+$. 
Then $P_+(E_0)$ is finite dimensional 
and the topological decomposition  
\[
E_0 = P_+(E_0) \oplus P_-(E_0)
\]
reduces $L$. Hence, $L = L_+ \oplus L_-$,
where $L_\pm$ is the part of $L$ in $P_{\pm}(E_0)$, respectively,
with the domains $D(L_\pm) = P_\pm(E_1)$. 
Moreover, $P_{\pm}$ decomposes $E_1$ %by the embedding $E_1 \hookrightarrow E_0$,  
and, without loss of generality, we can take the norm on $E_1$
so that $ \| v \|_{E_1} = \| P_+ v \|_{E_1} + \| P_- v \|_{E_1}.$
Note that
\begin{equation*}
\sigma(L_-) \subset [{\rm Re} \, z < \omega - \gamma],
\qquad 
\sigma(L_+) = \sigma^+ \subset [{\rm Re} \, z > \omega + \gamma],
\end{equation*}
which implies there is a constant $M_0 \ge 1$ such that
\begin{equation}
\label{estimate-sg}
\begin{split}
\|e^{L_{-} t} P_{-} \|_{\mathcal{L}(E_0)} &\le M_0 e^{(\omega - \gamma) t}, \\
\|e^{-L_{+}t} P_{+} \|_{\mathcal{L}(E_0)} &\le M_0 e^{-(\omega + \gamma) t}, \qquad t \ge 0
\end{split}
\end{equation}
where $\{ e^{L_{-}t} : t \ge 0 \}$ is the analytic semigroup in $P_-(E_0)$ generated by $L_{-}$ and
$\{ e^{L_{+}t} : t \in \mathbb{R} \}$ is the group in $P_+(E_0)$
generated by the bounded operator $L_+$.

From \cite[Theorem 5.2]{LeC11} one sees that 
$\big(\mathbb{E}_0(J), \mathbb{E}_1(J) \big)$
is a pair of maximal regularity for $-L$ and 
it is easy to see that $-L_-$ inherits the property of maximal regularity.
In particular, the pair $\big( P_- (\mathbb{E}_0(J)), P_-(\mathbb{E}_1(J)) \big)$
is a pair of maximal regularity for $-L_-$.
In fact, since $type( - \omega + L_{-} ) < - \gamma < 0$ 
we see that $\big( P_-(\mathbb{E}_0(\mathbb{R}_+)), P_-(\mathbb{E}_1(\mathbb{R}_+)) \big)$
is a pair of maximal regularity for $(\omega - L_-)$.
This, in turn, implies the a priori estimate
\begin{equation}
\label{estimate-stable}
\| e^{ - \omega t} w \|_{\mathbb{E}_1(J_T)}
\le M_1 \Big( \| w_0 \|_{E_{\mu}} + \| e^{- \omega t} f \|_{\mathbb{E}_0(J_T)}\Big)
\end{equation}
for $J_T := [0,T]$, any $T \in (0, \infty)$
(or $J_T = \mathbb{R}_+$ for $T = \infty$), 
with a universal constant $M_1 > 0$, where
$w$ is a solution of the linear Cauchy problem
\begin{equation*}
\begin{cases} \dot{w} - L_- w = f,\\ w(0) = w_0, \end{cases} \qquad \text{with} \quad (f, w_0) \in \Big( C\big( (0,T), P_- E_0 \big), P_- E_0 \Big).
\end{equation*}

\medskip

{\bf (ii)} By way of contradiction, suppose that $\rs$ is stable for \eqref{AMC}.
Then for every $\varepsilon > 0$ there exists a number $\delta > 0$ such that 
\eqref{reduced-2} admits
for each $r_0 \in \mathbb{B}_{E_{\mu}}(\rs, \delta)$ a global solution
$
r = r(\cdot, r_0) \in C^1_{1 - \mu}(\mathbb{R}_+, E_0) \cap C_{1 - \mu}(\mathbb{R}_+, E_1),
$
which satisfies 
\begin{equation}
\label{less-eps}
\| r(t) \|_{E_{\mu}} < \varepsilon, \qquad t \ge 0 \, .
\end{equation}
We can assume without loss of generality that $\beta$ and $\varepsilon$ are chosen such that
\begin{equation}
\label{beta}
2 C_0(M_0 + M_1 \gamma) \beta \le \gamma \qquad \text{and} \qquad \varepsilon \le \varepsilon_0(\beta),
\end{equation}
where $C_0 := \max \{\| P_- \|_{\mathcal{L}(E_0)}, \| P_+ \|_{\mathcal{L}(E_0)} \}$.
As $P_+(E_0)$ is finite dimensional, we may also assume that 
\begin{equation*}
\|P_+ v \|_{E_{\nu}} = \| P_+ v \|_{E_0}, \qquad v \in E_0, \quad \nu \in \{ \mu, 1 \},
\end{equation*} 
using the fact that $P_+ E_0 \subset D(L^n)$ for every $n \in \mathbb{N}$,
c.f. \cite[Proposition A.1.2]{LUN95}.

\medskip

{\bf CLAIM 1:} For any initial value $r_0 \in \mathbb{B}_{E_{\mu}}(\rs, \delta)$,
$P_+ r$ admits the representation
\begin{equation}
\label{P-plus-formula}
P_+ r(t) = - \int_t^{\infty} e^{L_+(t - s)} P_+ g( r(s) ) \, ds \qquad t \ge 0.
\end{equation}
For this we first establish that, for $r_0 \in \mathbb{B}_{E_{\mu}}(\rs, \delta)$,
\[
e^{- \omega t} r \in BC_{1 - \mu}(\mathbb{R}_+, E_1) := \left\{ u \in C((0, \infty), E_1): \sup_{t \in \mathbb{R}_+} t^{1 - \mu} \| u(t) \|_{E_1} < \infty \right\}.
\]
The mapping property 
$
g : \mathbb{E}_1(J_T) \rightarrow \mathbb{E}_0(J_T)
$ 
follows analogously to the mapping property derived for $\mathcal{G}$
in the proof of Theorem~\ref{Thm:Stability} above, $0 < T < \infty$. 
Together with the inequalities \eqref{estimate-g} and 
\eqref{estimate-stable}, this yields
\begin{equation}
\label{P-minus-1}
\begin{split}
&\| e^{-\omega t} P_{-} r \|_{B_{1 - \mu}(J_T, E_1)}
\le  M_1 \Big( \|P_{-} (r_0 - \rs) \|_{E_{\mu}}\\
& \qquad+ C_0 \beta \| e^{-\omega t} P_+ r \|_{B_{1 - \mu}(J_T, E_1)}
+ C_0 \beta \| e^{-\omega t} P_- r \|_{B_{1 - \mu}(J_T, E_1)}
\Big)
\end{split}
\end{equation}
for any $0 < T < \infty$.
Due to \eqref{beta}, we have $M_1 C_0 \beta \le 1/2$ and can further conclude
\begin{equation}
\label{estimate-Pminus-2}
\begin{split}
\| e^{-\omega t} P_{-} r \|_{B_{1 - \mu}(J_T, E_1)}
\le  2 M_1 \Big( \| P_{-} (r_0 - \rs) \|_{E_{\mu}}
+ C_0 \beta \| e^{-\omega t} P_+ r \|_{B_{1 - \mu}(J_T, E_1)}
\Big).
\end{split}
\end{equation}
It follows from \eqref{less-eps} that
\begin{equation*}
t^{1 - \mu} \| e^{-\omega t} P_+ r(t) \|_{E_1}
\le t^{1 - \mu} e^{-\omega t} C_0 \| r(t) \|_{E_{\mu}} \le C_0 C_1 \varepsilon
\end{equation*}
where $C_1 := \sup \{ t^{1 - \mu} e^{-\omega t} : t \ge 0 \} < \infty$.
Inserting this result into \eqref{estimate-Pminus-2} yields
\begin{equation}
\label{estimate-rho-1}
\| e^{-\omega t} r \|_{B_{1 - \mu}(J_T, E_1)}
\le 2 M_1 \| P_{-} (r_0 - \rs) \|_{E_{\mu}} + (2 M_1 C_0 \beta + 1) C_0 C_1 \varepsilon \le C_2
\end{equation}
for any $0 < T < \infty$.
However, since $T$ is arbitrary and 
\eqref{estimate-rho-1} is independent of $T$ we conclude that 
$e^{-\omega t} r \in BC_{1 - \mu}(\mathbb{R}_+, E_1)$, for any initial value 
$r_0 \in \mathbb{B}_{E_{\mu}} (\rs, \delta)$.
Next we note that, for $s \ge t$, by \eqref{estimate-sg}
\begin{equation}
\label{estimate-integral-1}
\begin{split}
\| e^{L_+(t-s)} P_+ g(r(s)) \|_{E_0}
&\le M_0 C_0 \beta e^{(\omega + \gamma)(t-s)} \| r(s) \|_{E_1}\\
&\le M_0 C_0 \beta e^{\omega t} e^{\gamma (t-s)} s^{\mu - 1} \| e^{-\omega s} r \|_{B_{1 - \mu}(\mathbb{R}_+, E_1)},
\end{split}
\end{equation}
by which the integral in \eqref{P-plus-formula} exists for $t \ge 0$,
convergence in $E_1$. Moreover, 
\begin{equation}
\label{estimate-integral-2}
\begin{split}
\left\| \int_t^{\infty} e^{L_+ (t-s)} P_+ g(r(s)) \, ds \right\|_{E_0}
\le e^{\omega t} M_0 C_0 C_3 \beta \| e^{-\omega t} r \|_{B_{1 - \mu}(\mathbb{R}_+, E_1)},
\end{split}
\end{equation}
where $C_3 := \sup \big\{ \int_t^{\infty} e^{\gamma (t-s)} s^{\mu - 1} \, ds : t \ge 0 \big\} < \infty$.
Noting that $w = P_+ r$ solves the Cauchy problem
\begin{equation*}
\begin{cases} \dot{w} - L_+ w = P_+ g(r),\\ w(0) = P_+ (r_0 - \rs), \end{cases}
\end{equation*}
it follows from the variation of parameters formula that, for $t \ge 0$ and $\tau > 0$,
\begin{equation*}
\label{plus-unstable}
P_+ r(t) =
e^{L_+ (t - \tau)} P_+ r(\tau) + \int_{\tau}^t e^{ L_+(t - s)} P_+ g(r(s)) \, ds.
\end{equation*}
This representation holds for any $\tau > 0$ and the claim 
follows from \eqref{estimate-sg} and \eqref{less-eps} by sending $\tau$ to $\infty$.

\medskip

{\bf CLAIM 2:} If $r_0 \in \mathbb{B}_{E_{\mu}}(\rs, \delta)$
and $\| r(t) \|_{E_{\mu}} < \varepsilon$ for all $t \geq 0$, 
then
\[
\| P_+ (r_0 - \rs) \|_{E_{\mu}} \leq 2 M_0 M_1 C_3 \| P_- (r_0 - \rs) \|_{E_{\mu}}.
\]
From \eqref{P-plus-formula} and \eqref{estimate-integral-1} follows
\begin{equation}
\label{P-plus-2}
\begin{split}
&\| e^{-\omega t} P_+ r \|_{B_{1 - \mu}(\mathbb{R}_+, E_0)}\\
&\le \frac{M_0 C_0 \beta}{\gamma} 
\Big( \| e^{-\omega t} P_+ r \|_{B_{1 - \mu}(\mathbb{R}_+, E_1)}
+ \| e^{-\omega t} P_- r \|_{B_{1 - \mu}(\mathbb{R}_+, E_1)} \Big)
\end{split}
\end{equation}
where we have used the fact that
$\sup_{t \ge 0} \big\{ t^{1 - \mu} \int_t^{\infty} e^{\gamma (t-s)} s^{\mu - 1} \, ds \big\} \le 1/\gamma$.
Adding the estimates in \eqref{P-minus-1} and \eqref{P-plus-2} and 
employing \eqref{beta} yields
\begin{equation}
\label{B}
\|e^{-\omega t} r \|_{B_{1 - \mu}(\mathbb{R}_+, E_1)} \le 2 M_1 \| P_- (r_0 - \rs) \|_{E_{\mu}}.
\end{equation}
The representation \eqref{P-plus-formula} in 
conjunction with \eqref{estimate-integral-2} and \eqref{B} then implies
\begin{equation}
\begin{split}
\|P_+ (r_0 - \rs) \|_{E_{\mu}}
\le M_0 C_0 C_3 \beta \| e^{-\omega t} r \|_{B_{1 - \mu}(\mathbb{R}_+, E_1)}
 \le M_0 C_3 \|P_- (r_0 - \rs) \|_{E_\mu},
\end{split}
\end{equation}
where the last inequality follows from $2 C_0 M_1 \beta \le 1$.
We have thus demonstrated the claim, and the theorem follows 
by way of contradiction. In particular, note that if $r_0 \in \mathbb{B}_{E_{\mu}}(\rs, \delta)$
is chosen with $\| P_- (r_0 - \rs) \|_{E_{\mu}} = 0$, then it must hold that 
$\| P_+ (r_0 - \rs) \|_{E_{\mu}} = 0$, so that $r_0 = \rs$, which contradicts
the stability assumption.
\end{proof}

%%%%%%%%%%%%%%%%%%%%%%%%%%%%%%%%%%%%%%%%%%%%%%%%%%%%%%%%%%%%%%%%%%%%%%%%%%%%%%%%%%%%%%%%%%%%%%%%%%%%%%%%%%%%%%%%%%%
\section{Bifurcation From Cylinders}\label{Sec:Bifurcation}
%%%%%%%%%%%%%%%%%%%%%%%%%%%%%%%%%%%%%%%%%%%%%%%%%%%%%%%%%%%%%%%%%%%%%%%%%%%%%%%%%%%%%%%%%%%%%%%%%%%%%%%%%%%%%%%%%%%

We conclude by investigating the interactions between the family of cylinders,
which we will informally consider as the \emph{trivial} equilibria of \eqref{AMC},
and the families of undulary curves, the \emph{non--trivial} equilibria. 
In particular, by restricting \eqref{ReducedAMC} to a problem on profile functions which
are symmetric about the $y$--axis (i.e. \emph{even} functions on $\Tone$), we observe
subcritical bifurcations from the family of cylinders which occur at the cylinder
of radius $1 / \ell$, for any $\ell \in \mathbb{N}$, where the reciprocal of the radius
$\rs$ is taken as a bifurcation parameter. Most of the terminology and notations 
employed in this section coincide with those of Kielh\"ofer \cite{KIE12}.

To begin, define the operator
\[
\bG(\tr, \lambda) := P_0 G(\psi_{\star}(\tr, \rs)), \qquad \lambda := 1 / \rs, \; \rs > 0,
\]
and we immediately note that
\[
\bG \in C^{\omega} \left( (h^{2 + \alpha}_{0,e}(\Tone) \cap U_0) \times (0, \infty), h^{\alpha}_{0, e}(\Tone) \right),
\]
where $F_{j, e} := h^{2j + \alpha}_{0, e}(\Tone)$ denotes the space
of even, zero--mean functions with $h^{2j + \alpha}$--regularity over $\Tone$, $j = 0, 1$.
%in some small neighborhood $U_0 \cap F_{1,e}$ where $\psi_{\star}$ is defined.
We observe that $\bG (\cdot, \lambda)$ preserves symmetry about the $y$--axis, a property
that follows quickly from the definition \eqref{Eqn:G} of the governing operator $G$
and the representation of $\psi_{\star}$ derived in Remarks~\ref{Rem:PsiProperties}(f).
The fact that $G$ preserves even functions is easily noted by the fact that every
term of $G(r)$ is either constant on $\Tone$ due to integration, 
or depends only on $r$, $r_{xx}$, and $r_x^2$, which are each even
functions on $\Tone$ when $r$ is itself taken to be even on $\Tone$. 

\begin{thm}\label{Thm:Bifurcation}
Fix $\ell \in \mathbb{N}$. Then $(0, \ell)$ is a bifurcation point for the equation
\begin{equation}\label{BifurcEqn}
\bG(\tr, \lambda) = 0, \qquad (\tr, \lambda) \in h^{2 + \alpha}_{0, e}(\Tone) \times (0, \infty).
\end{equation}
In particular, there exists a positive
constant $\delta_{\ell} > 0$ and a nontrivial analytic curve
\begin{equation}\label{Eqn:BifurcatingBranch}
\left\{ (\tr_{\ell}(s), \lambda_{\ell}(s)) \in h^{2 + \alpha}_{0,e}(\Tone) \times \mathbb{R} : s \in (- \delta_{\ell}, \delta_{\ell}), (\tr_{\ell}(0), \lambda_{\ell}(0)) = (0, \ell) \right\},
\end{equation}
such that 
\[
\bG(\tr_{\ell}(s), \lambda_{\ell}(s)) = 0 \qquad \text{for all} \quad s \in (-\delta_{\ell}, \delta_{\ell}),
\]
and all solutions of \eqref{BifurcEqn} in a neighborhood of $(0, \ell)$ are either a
trivial solution $(0, \lambda)$ or an element of the nontrivial curve \eqref{Eqn:BifurcatingBranch}.
Moreover, if $\lambda \in \mathbb{R}_+ \setminus \mathbb{N}$, then $(0, \lambda)$ is not a bifurcation
point for \eqref{BifurcEqn}. We can further conclude that
\begin{enumerate}
	\item lifting the curve \eqref{Eqn:BifurcatingBranch}, via $\psi_{\star}$, we get elements of the family
	of $2 \pi / \ell$--periodic undulary curves with parameter values $|B| < \tilde{\delta}_{\ell}$
	for some $\tilde{\delta}_{\ell} > 0$.
		
	\item the bifurcation is a subcritical pitchfork type bifurcation. More precisely, for all $\ell \in \mathbb{N}$,
	we have 
	\[
	\dot{\lambda}_{\ell}(0) = 0 \quad \text{and} \quad \ddot{\lambda}_{\ell}(0) < 0,
	\]
	where ``$\; \, \dot{ } \;$'' denotes differentiation with respect to the parameter $s$. 
	
	\item the bifurcating branch of undulary curves are unstable equilibria of \eqref{AMC},
	at least for parameter values $|B| < \tilde{\delta}_{\ell}$ sufficiently small.
\end{enumerate}
\end{thm}

\begin{proof}
By working in the setting of even functions on $\Tone$, we can take advantage of the 
cosine--Fourier series representation
\[
\tr(x) = \sum_{k \ge 1} \hat{\tr}(k) \cos (kx), \qquad \text{for all} \quad \tr \in h^{\sigma}_{0, e}(\Tone),
\] 
where $\hat{\tr}(k) = \frac{1}{\pi} \int_{\Tone} \tr(x) \cos (kx) dx$ are the cosine--Fourier
coefficients of $\tr \in h^{\sigma}_{0, e}(\Tone),$ $k \ge 1$. It follows easily from \eqref{Eqn:DmGat0} 
that the linearization $D_1 \bG(0, \lambda)$ is a Fourier multiplier with symbol
\[
( M_k )_{k \ge 1} := \left( \lambda^2 - k^2 \right)_{k \ge 1}.
\]
We then see that $D_1 \bG(0, \lambda)$ is bijective from
$F_{1, e}$ onto $F_{0, e}$ when  
$\lambda \in \mathbb{R}_+ \setminus \mathbb{N}$,
so that only points of the form $(0, \ell)$ can 
possibly be bifurcation points for \eqref{BifurcEqn}.

Proceeding to verify that $(0, \ell)$ is indeed a bifurcation point, 
we compute the kernel and range of $D_1 \bG(0, \ell)$ as
\[
N_{\ell} = \text{span} \{ \cos (\ell x) \} \quad \text{and} \quad R_{\ell} = \overline{\text{span}} \{ \cos (kx): k \ge 1, k \not= \ell \},
\]
respectively. Since $h^{\sigma}(\Tone) \hookrightarrow L_2(\Tone)$, 
we can borrow the $L_2$-inner product to realize $N_{\ell}$ as a 
topological complement to $R_{\ell}$ as subspaces of $F_{1,e}$.
By compactness of the resolvent $R(\lambda) := (\lambda - DG(\rs))^{-1}$, 
$\lambda \in \rho(DG(\rs))$, it follows that 
$D_1 \bG(0, \ell)$ is a Fredholm operator of index zero.
Then, defining the element $\hat{v}_0 := \cos(\ell x),$ and noting that
\[
D^2_{1 2} \bG(0, \ell) \hat{v}_0 = - 2 \ell P_0 \hat{v}_0 = - 2 \ell \cos(\ell x) \notin R_{\ell},
\]
it follows by \cite[Theorem 1.7]{CR71}, or \cite[Theorem I.5.1]{KIE12},
that \eqref{BifurcEqn} bifurcates at $(0, \ell)$, which proves the main
part of the theorem.

{\bf (a)} Notice that $\psi_{\star}(0, \ell)$ is precisely the constant
function $\rs = 1 / \ell$, and, utilizing the explicit characterization
of equilibria from Section~\ref{Sec:Equilibria} above, we know that a curve
of non--trivial equilibria for \eqref{AMC} containing this function must coincide with the 
family of $\frac{2 \pi}{\ell}$--periodic unduloids. Hence, it suffices to show
that $\bG(\tr, \lambda) = 0$ if and only if $G(\psi_{\star}(\tr, \lambda^{-1})) = 0$,
i.e. $\psi_{\star}(\tr, \lambda^{-1})$ is an equilibrium of \eqref{AMC} if 
and only if $(\tr, \lambda)$ solves \eqref{BifurcEqn}. However, this follows 
immediately from the relation
\begin{equation}\label{Eqn:GLinear1}
G(r) = D_1 \psi_{\star}(P_0 r, \eta) P_0 G(r), \qquad r \in E_1,
\end{equation}
which was justified in deriving \eqref{Eqn:GTangent} above.

{\bf (b)} We will follow the characterization of bifurcation types as developed
in \cite[Sections I.6 and I.7]{KIE12}. Computing the second derivative
\begin{align*}
D^2_{1 1} \bG(0, \ell) [\hat{v}_0, \hat{v}_0] &= P_0 \bigg( 2 \ell^4 Q_0(\cos^2 (\ell x)) - 2 \ell^3 P_0 (\cos^2 (\ell x)) - \ell^4 Q_0(\sin^2 (\ell x))\\
& \qquad - \ell^3 \sin^2 (\ell x) + 2 \ell^3 Q_0( \cos^2 (\ell x)) \bigg)\\
%&= P_0 \bigg( 2 \ell^4 Q_0 (\cos^2 (\ell x)) - 2 \ell^3 \cos^2 (\ell x) + 4 \ell^3 Q_0( \cos^2 (\ell x))\\
%& \qquad - \ell^4 Q_0( \sin^2 (\ell x)) - \ell^3 \sin^2 (\ell x) \bigg)\\
%&= P_0 \bigg( \underbrace{\ell^3 \Big( 1 + (2 \ell + 4) Q_0( \cos^2 (\ell x)) - \ell Q_0 (\sin^2 (\ell x)) \Big)}_{\in \mathbb{R}} - \ell^3 \cos^2 (\ell x) \bigg)\\
&= - \ell^3 P_0 \left( \frac{ \cos(2 \ell x) + 1}{2} \right) = - \ell^3 / 2 \; \cos(2 \ell x),
\end{align*}
we note that $- \ell^3 / 2 \; \cos(2 \ell x) \in R_{\ell}$, from which it follows
that $\dot{\lambda}(0) = 0$. 
Meanwhile, utilizing the representation \cite[(I.6.11) and (I.6.9)]{KIE12} 
and following a considerable amount of computation, one will see that
\[
\ddot{\lambda}(0) = -\frac{1}{3} \frac{\langle 6 \ell^4 \hat{v}_0, \hat{v}_0 \rangle}{\langle 2 \ell \hat{v}_0, \hat{v}_0 \rangle} = - \ell^3 < 0,
\]
where $\langle \cdot, \cdot \rangle$ denotes the inner product on $L^2(\Tone)$,
within which $h^{\sigma}_{0, e}(\Tone)$ is embedded.

{\bf (c)} %Following from the observed bifurcations, we are able to 
We can now track the so--called
{\em critical eigenvalue} $\mu_{\ell}(\lambda)$ of the linearization $D_1 \bG(0, \lambda)$,
which is the eigenvalue which passes through $0$ at $\lambda = \ell$, with non--vanishing speed,
as guaranteed by the observed bifurcation. Moreover, via eigenvalue perturbation techniques, 
we can also track the perturbed eigenvalues $\hat{\mu}_{\ell}(s)$ of the linearization
$D_1 \bG(\tr_{\ell}(s), \lambda_{\ell}(s))$. Then, by taking a derivative of the 
relation \eqref{Eqn:GLinear1}, we observe that 
\begin{equation}\label{Eqn:GLinear2}
D G(\psi_{\star}(\tr, \eta)) D_1 \psi_{\star}(\tr, \eta) \tilde{h} = D_1 \psi_{\star}(\tr, \eta) D_1 \mathcal{G}(\tr, \eta) \tilde{h}, \qquad \tilde{h} \in F_1,
\end{equation}
from which one will easily conclude that $\hat{\mu}(s)$ must also be an eigenvalue of the
linearization $D G(\psi_{\star}(\tr_{\ell}(s), \lambda_{\ell}^{-1}(s))$.
Finally, by the subcritical structure of the bifurcation, 
we conclude that, for sufficiently small $|s| < \delta_{\ell}$, the perturbed
eigenvalue $\hat{\mu}_{\ell}(s)$ has positive real part, from which instability follows by
a similar argument to Theorem~\ref{Thm:Instability} above.   
\end{proof}

\section*{Acknowledgements}
\noindent I would like to thank Professor Gieri Simonett for introducing me to geometric
evolution equations and for providing many helpful conversations and suggestions
in the preparation of this manuscript.

\end{document}